\documentclass[9pt,reqno]{amsart}
\allowdisplaybreaks
\usepackage{amsmath,amstext,amssymb,epsfig}
\usepackage{tikz-cd}
\usepackage{graphicx}
\usepackage{mathrsfs}
\usepackage{xcolor}
\usepackage[a4paper, margin=1in]{geometry} % adjust as necessary
\calclayout
\makeatletter
\renewcommand{\@seccntformat}[1]{\bf\csname the#1\endcsname.}
\renewcommand{\section}{\@startsection{section}{1}
	\z@{.7\linespacing\@plus\linespacing}{.5\linespacing}
	{\normalfont\upshape\bfseries\centering}}
\renewcommand{\@biblabel}[1]{\@ifnotempty{#1}{#1.}}
\makeatother
\theoremstyle{plain}
\newtheorem{thm}{Theorem}[section]
\newtheorem{lem}[thm]{Lemma}
\newtheorem{prop}[thm]{Proposition}
\newtheorem{cor}[thm]{Corollary}

\theoremstyle{definition}
\newtheorem{ex}[thm]{Example}
\newtheorem{defn}[thm]{Definition}
\newtheorem{rem}{Remark}[section]

\usepackage{tikz-cd}
\usepackage[parfill]{parskip}
\usepackage[german]{varioref}
\usepackage[all]{xy}
\usepackage{cancel}
\usepackage{stmaryrd}

\usepackage{tikz}
\usetikzlibrary{shapes.geometric, arrows, positioning}

\tikzstyle{startstop} = [rectangle, rounded corners, minimum width=3cm, minimum height=1cm,text centered, draw=black, fill=red!30]
\tikzstyle{process} = [rectangle, minimum width=3cm, minimum height=1cm, text centered, draw=black, fill=orange!30]
\tikzstyle{decision} = [diamond, minimum width=3cm, minimum height=1cm, text centered, draw=black, fill=green!30]
\tikzstyle{arrow} = [thick,->,>=stealth]
\usepackage{float}

\def\R{{\mathcal R}}
\def\L{{\mathcal L}}
\def\N{{\mathcal N}}

\def\Q{{\mathcal Q}} 
\def\E{{\mathcal E}}
\def\H{{\mathcal H}}

\def\>{\succ}
\def\<{\prec}

\def\M{{\mathcal M}}

\def\b{\beta}
\def\a{\alpha}
\def\l{\lambda}
\def\p{\partial}

\def\m{\mu}

\def\E{\mathcal{E}}
\begin{document}
	\title[Sania Asif \textsuperscript{1}%, Zhixiang Wu\textsuperscript{2}
    ]{Cohomology, Homotopy, Extensions, and Automorphisms of Nijenhuis~Lie Conformal Algebras}
%%%%%%%%%%%%%%%%%%%%%%%%%%%%%%%%%%%%%%%%%%%%%%%%%%%%%%%%%%%%%%%%%%%%%%%%%%%%%%%%
\author{Sania Asif\textsuperscript{1}%, Zhixiang Wu\textsuperscript{2}
}
\address{\textsuperscript{1} Institute of Mathematics, Henan Academy of Sciences, Zhengzhou, 450046, P.R. China.}
%\address{\textsuperscript{2}Department of Mathematics, Zhejiang University, Hangzhou, Zhejiang Province, 310027, P.R. China} 
\email{\textsuperscript{1}11835037@zju.edu.cn}
%\email{\textsuperscript{2}wzx@zju.edu.cn}
\date{\today}
%This work is supported by the Jiangsu Natural Science Foundation Project (Natural Science Foundation of Jiangsu Province), Relative Gorenstein cotorsion Homology Theory and Its Applications (No.BK20181406) 
\thanks{This work is funded by the Second batch of the Provincial project of the Henan Academy of Sciences (No. 241819105).}% and by NNSFC (No.12471038, No.12171129)
%%%%%%%%%%%%%%%%%%%%%%%%%%%%%%%%%%%%
\begin{abstract}
This paper explores various algebraic and homotopical aspects of Nijenhuis Lie conformal algebras, including their cohomology theory, $\L_\infty$-structures, non-abelian extensions, and automorphism groups. We define the cohomology of a Nijenhuis Lie conformal algebra and relate it to the deformation theory of such structures. We also introduce $2$-term Nijenhuis $\L_\infty$-conformal algebras and establish their correspondence with crossed modules and $3$-cocycles in the cohomology of Nijenhuis Lie conformal algebras. Furthermore, we develop a classification theory for non-abelian extensions of Nijenhuis Lie conformal algebras via the second non-abelian cohomology group. Finally, we study the inducibility problem for automorphisms under such extensions, introducing a Wells-type map and deriving an associated exact sequence.
\end{abstract}\footnote{*The corresponding author's email: 11835037@zju.edu.cn.}
\maketitle 
\tableofcontents
\section{ Introduction} \textbf{C}onformal algebras were introduced by Victor Kac in \cite{kac} as an algebraic framework to formalize the operator product expansion (OPE) of chiral fields in two-dimensional conformal field theory (CFT). These algebras encode the singular part of the OPE in terms of Fourier coefficients and provide a natural setting for studying infinite-dimensional Lie algebras with a $\mathbb{C}[\partial]$-module structure, where $\partial$ represents the infinitesimal translation operator. They have since become essential tools in the study of vertex operator algebras, quantum field theory, and integrable systems. Lie conformal algebras, in particular, offer a suitable framework for understanding infinite-dimensional Lie algebras satisfying locality conditions. The structure theory of these objects was systematically developed in works such as \cite{HB, KP, L}, laying the foundation for further investigations into their representation theory, deformations, and cohomology. In recent years, there has been growing interest in enriched structures on Lie conformal algebras, particularly those involving operators that govern deformations and integrability.

A key such operator is the \textbf{N}ijenhuis operator, first introduced by Albert Nijenhuis in \cite{NA} in the context of differential geometry to study integrability conditions of almost complex structures. Algebraically, a Nijenhuis operator on an algebra $(\L, \circ)$ is a linear map $\mathcal{N}: \L \to \L$ satisfying the identity:
$$
\mathcal{N}(p) \circ \mathcal{N}(q) = \mathcal{N}\left( \mathcal{N}(p) \circ q + p \circ \mathcal{N}(q) - \mathcal{N}(p \circ q) \right), \quad \forall~ p, q \in \L.
$$
This condition ensures that the deformed multiplication ${[p,q]}_{\mathcal{N}} := [\mathcal{N}(p),q] + [p, \mathcal{N}(q)] - \mathcal{N}([p,q])$ remains a Lie bracket whenever $\mathcal{N}$ acts on a Lie algebra. Such operators naturally appear in integrable systems and are fundamental in algebraic deformation theory, where they yield canonical infinitesimal deformations of associative and Lie algebras \cite{G}. Notably, the Newlander-Nirenberg theorem establishes that an almost complex structure on a smooth manifold is integrable if and only if it is a Nijenhuis operator \cite{NN}. Recent developments have extended the theory of Nijenhuis operators to various classes of algebras, including Rota-Baxter algebras, pre-Lie algebras, and Lie conformal algebras \cite{A, AWang, AWW, AW, BR}. In this paper, we focus on Nijenhuis Lie conformal algebras that is a Lie conformal algebras equipped with a Nijenhuis operator, and investigate several foundational aspects like cohomology, homotopy theory, non-abelian and abelian extensions, and automorphism groups with inducibility conditions.

\textbf{C}ohomology provides a powerful tool for analyzing algebraic structures and their deformations. For a given Lie conformal algebra $\L$ and its module $\M$, one defines a cochain complex $C^\bullet(\L, \M)$, together with a coboundary operator $d$ satisfying $d^2 = 0$. The resulting cohomology groups $H^\bullet(\L, \M)$ capture crucial structural information, such as derivations, central extensions, and obstructions to deformations. Classical analogues include Chevalley-Eilenberg cohomology for Lie algebras and Hochschild cohomology for associative algebras. In the context of conformal algebras, cohomological techniques have been employed to study formal deformations and operadic structures \cite{AWW, AWMB, STZ}. In Section $3$, we develop the cohomology theory of Nijenhuis Lie conformal algebras and show how it governs the deformation theory of such structures.

\textbf{H}omotopy theory offers a higher categorical perspective on algebraic structures through the notion of $\infty$-algebras. An $\L_\infty$-algebra, or strongly homotopy Lie algebra, generalizes Lie algebras by allowing brackets of arbitrary arity subject to generalized Jacobi identities up to coherent homotopy. These structures originated in rational homotopy theory and have found applications in deformation quantization, string field theory, and derived geometry \cite{LS, S}. In particular, $2$-term $\L_\infty$-algebras correspond to strict Lie 2-algebras, which are categorified versions of Lie algebras. As shown in \cite{BC}, strict $2$-term $\L_\infty$-algebras are equivalent to crossed modules of Lie algebras, while skeletal ones are classified by $3$-cocycles in the cohomology of the base Lie algebra. These ideas have been extended to the realm of conformal algebras in \cite{DS, SD}, where $\L_\infty$-conformal algebras were introduced as homotopic analogues of Lie conformal algebras. In Section 4, we define $2$-term Nijenhuis $\L_\infty$-conformal algebras and explore their connections to crossed modules and cohomology.

\textbf{N}on-abelian \textbf{e}xtensions provide a means of constructing new algebraic structures from existing ones. Originating in group theory with Eilenberg and MacLane \cite{EM}, and later extended to Lie algebras by Hochschild \cite{HS}, non-abelian extensions have recently seen renewed interest in the context of Leibniz algebras, Rota-Baxter algebras, and Lie groups \cite{D, F, WZ}. In this work, we define the second non-abelian cohomology group $H^2_{\text{nab}}(\L_\N, \H_\Q)$ for Nijenhuis Lie conformal algebras $\L_\N$ and $\H_\Q$, and show that equivalence classes of non-abelian extensions are completely classified by this cohomology group. We prove that different choices of sections preserve the equivalence relation among 2-cocycles, establishing a robust framework for classifying such extensions.

Finally, we address the inducibility problem for \textbf{a}utomorphisms of Nijenhuis Lie conformal algebras. Inspired by Wells' work on group extensions \cite{Wells}, and further developed in \cite{PSY}, we construct a Wells-type map
$$
\mathfrak{W} : \mathrm{Aut}(\H_\Q) \times \mathrm{Aut}(\L_\N) \to H^2_{\text{nab}}(\L_\N, \H_\Q)
$$
associated to any non-abelian extension $0 \to \H_\Q \to \E_\R \to \L_\N \to 0$. Our main result, Theorem \ref{thm6.4}, asserts that a pair of automorphisms $(\alpha, \beta)$ is inducible if and only if the image of the Wells map vanishes, indicating that the obstruction lies in the second non-abelian cohomology group. Furthermore, in Theorem \ref{thm6.5}, we derive a Wells-type exact sequence linking automorphism groups and cohomology, offering a deeper understanding of the structure of non-abelian extensions in the context of Nijenhuis Lie conformal algebras.

This paper is organized as follows: In Section \textbf{2}, we recall basic definitions and properties of Nijenhuis operators on Lie conformal algebras.  In Section \textbf{3}, we develop the cohomology theory of Nijenhuis Lie conformal algebras. Section \textbf{4} introduces $2$-term Nijenhuis $\L_\infty$-conformal algebras and explores their connections to crossed modules and cohomology. Section \textbf{5} is devoted to non-abelian extensions and their classification via second non-abelian cohomology. We also provide abelian extensions as a special case of non-abelian extensions. Finally, in Section \textbf{6}, we study automorphisms and the inducibility problem using the Wells map and associated exact sequence.
\par All vector spaces and algebras considered in this paper are over the field $\mathbb{C}$ of complex numbers unless otherwise specified.

 \section{ Preliminaries of Nijenhuis operators on Lie conformal algebra}
 In this section, we define Nijenhuis Lie conformal algebra structure, followed by the introduction of Nijenhuis operator $\N$ on Lie conformal algebras $\L$. We then provide a Nijenhuis Lie conformal algebra structure on the direct sum of two Lie conformal algebras. Additionally, we define conformal representation on $\L$ and give the notion of relative Nijenhuis operator on $\L$. Finally, we define the crossed module of Lie conformal algebra, which is used in Section $4$ to study the crossed module of Nijenhuis Lie conformal algebras and their relationship with the Nijenhuis $\L_\infty$-conformal algebras.
\begin{defn} A pair $(\rho, \M)$ consisting of a $\mathbb{C}$-linear map $\rho: \L \to Cend(\M)$ and a $\mathbb{C}[\p]$-module $\M$, is called representation of a Lie conformal algebra $(\L, [\cdot_\l\cdot])$, if it satisfies the following axioms:\begin{align*} \rho(\partial p)_\l = -\l \rho(p)_\l,\quad & \rho \partial= \partial \rho,\\ [\rho(p)_\l \rho(q)_\mu] &= \rho([p_{\l}q])_{\l+\mu},\\ \rho([p_{\l}q])_{\l+\mu}r &= \rho( p)_{\l}(\rho(q)_\mu r)-
	\rho(q)_{\mu}(\rho(p)_{\l} r).
	\end{align*} for all $p,q,r\in\L$.
\end{defn}
Let $(\rho, \M)$ be a representation of the Lie conformal algebra $(\L, [ \cdot_\l\cdot ])$. Then the $\l$-bracket ${[\cdot_\l\cdot]}_{\rho}$ on the direct sum of $\mathbb{C}[\p]$-modules $\L\oplus \M$ is defined by
\begin{align*}{[(p + m)_{\l}(q + n)]}_{\rho}=[p_\l q]+ \rho(p)_\l n - \rho(q)_{-\partial- \l}m, \quad p,q\in\L \text{ and }m, n\in \M, \end{align*} forms a Lie conformal algebra. This Lie conformal algebra $(\L \oplus \M, {[\cdot_\l\cdot ]}_{\rho})$ is called the semi-direct product Lie conformal algebra, denoted by $\L \ltimes \M$. 
 \begin{defn}\label{Nijenhuis}A $\mathbb{C}[\p]$-linear map $\N : \L \to \L$ is called a Nijenhuis operator on Lie conformal algebra $(\L,[\cdot_\l\cdot])$, if following equation holds:
\begin{align*} 
 [{\N(p)}_{\l} \N(q)]= \N\Big([{\N(p)}_{\l}q]+ [p_{\l}\N(q)]- \N([p_{\l}q])\Big), \quad \forall p,q\in \L. \end{align*}\end{defn} 
 A Lie conformal algebra $(\L,[\cdot_\l\cdot])$ along with Nijenhuis operator $\N$ is called Nijenhuis Lie conformal algebra, denoted by $(\L,[\cdot_\l \cdot], \N)$.
 \begin{defn}
 Given two Nijenhuis Lie conformal algebra $\L_\N$ and $\L'_{\N'}$, a homomorphism of Nijenhuis Lie conformal algebras is a conformal algebra homomorphism $\psi : (\L, [\cdot_\l \cdot])\to (\L', [\cdot_\l \cdot]')$ that satisfies $\psi\circ \N = \N'\circ \psi$.
 \end{defn}
 %There are a lot of Nijenhuis Lie conformal algebras. {\color{red}\begin{ex} If $\L$ is a Lie conformal algebra and direct sum of the $n$-copies of $\L$ $(\textit{given by }\L\oplus \L\oplus \L\oplus \cdots \oplus \L)$ defines a Lie conformal algebra with the $\l$-bracket :\begin{align*} [(x_1, x_2,\cdots,x_n)_\l&(a_1,\cdots a_n)]\\&:=([{x_1}_\l {a_1}]_\L , [{x_1}_\l {a_2}]_\L- [{a_1}_{-\p-\l}x_2]_\L ,\cdots, \overset{\text{i-th place for $i\geq 2$}}{[{x_1}_\l {a_i}]_\L- [{a_1}_{-\p-\l}x_i]_\L},\cdots [{x_1}_\l {a_n}]_\L- [{a_1}_{-\p-\l}x_n]_\L).\end{align*} With the above Lie conformal bracket, $\mathbb{C}[\p]$-module homomorphisms $\N, \N_i:\L \oplus \L \oplus \cdots \oplus \L\to \L\oplus \L\oplus \cdots \oplus \L$ defined by $\N(x_1,x_2,\cdots,x_n)= (\sum_{i=2}^{n} x_i,0,\cdots,0)$ and $\N_i(x_1,x_2,\cdots,x_n)= (x_i,0,\cdots,0)$, for $i\geq 2$ are Nijenhuis operator on the $nth$-direct sum Lie conformal algebras. \end{ex} }
 \begin{ex}
 The identity map $Id:\L \to \L$ is a Nijenhuis operator on any Lie conformal algebra $(\L,[\cdot_\l \cdot])$. Thus corresponding Nijenhuis Lie conformal algebra is denoted by $(\L,[\cdot_\l\cdot], Id )$.\end{ex}
 \begin{prop} Let $\L$ be a Lie conformal algebra, and $\N$ be a Nijenhuis operator. Then $(\L,[\cdot_\l\cdot], \N)$ yield another Nijenhuis Lie conformal algebra with the new Lie bracket given by $[p_\l q]_\N= [\N(p)_{\l}q]+ [p_{\l}\N(q)]- \N([p_{\l}q])$ and the same Nijenhuis operator $\N$, denoted by $\L_\N= (\L, {[\cdot_\l\cdot]}_\N)$. This Nijenhuis Lie conformal algebra is also called the induced Nijenhuis Lie conformal algebra. Moreover, $\N$ is a morphism of Nijenhuis Lie conformal algebras from $(\L, [\cdot_\l\cdot])$ to $ (\L, {[\cdot_\l\cdot]}_\N)$. 
 \end{prop}
Additionally, we have the following results for the Nijenhuis operators on Lie conformal algebras. \begin{prop}
Consider the Nijenhuis operator $\N : \L \to \L$ on the Lie conformal algebra $(\L, [ \cdot_\l\cdot ])$. Given that $k,l\geq 0$, following results hold:
\begin{enumerate}\item $\N^k$ is a Nijenhuis operator on the Lie conformal algebra $(\L, [ \cdot_\l\cdot ])$.
 \item $\N^l$ is a Nijenhuis operator on the Lie conformal algebra $(\L, {[ \cdot_\l\cdot ]}_{\N^{k}})$.
 \item The brackets ${[ \cdot_\l\cdot ]}_{\N^{k+l}}$ and $({[ \cdot_\l\cdot ]}_{\N^{k}})_{\N^{l}}$ are same. Here $({[ \cdot_\l\cdot ]}_{\N^{k}})_{\N^{l}}$ is the bracket obtained from ${[ \cdot_\l\cdot ]}_{\N^{k}}$ deformed by $\N^l$. 
 \item The map $\N^l : \L \to \L $ is a Lie conformal algebra morphism from $(\L, {[ \cdot_\l\cdot ]}_{\N^{k}})\to (\L, {[ \cdot_\l\cdot ]}_{\N^{k+l}}).$
 \item The Lie conformal algebras $(\L, {[\cdot_\l\cdot]}_{\N^{k}})$ and $(\L, {[ \cdot_\l\cdot ]}_{\N^{l}})$ are compatible.
	\end{enumerate}
\end{prop}
%{\color{red}\begin{proof}MAYREMOVEPROOFThe proof is exactly similar to the classical Lie algebra case \cite{Kosmann-Schwarzbach}.\end{proof}}

 \section{ Cohomology theory of Nijenhuis operators and Nijenhuis Lie conformal algebras}This section is further divided into various subsections, where we studied cohomology of Lie conformal algebras, Cohomology of Nijenhuis operators, and cohomology of Nijenhuis Lie conformal algebras in detail.
 \subsection{ Cohomology of the Lie conformal algebra} Let's review the cohomology of the Lie conformal algebra, which has coefficients in the representation $(\M,\rho)$. The space of $nth$-cochain $C^{n}(\L,M)$ contains $\mathbb{C}$-linear maps of the form $f_{\l_1,\cdots\l_{n-1}} :\wedge^{\otimes n} \L\to \M[\l_1,\cdots\l_{n-1}] $, that satisfy conformal sesquilinearity, and conformal skew-symmetry conditions, given in \cite{AWaveraging}.
The coboundary map associated to this cochain in defined by $\delta : C^{n}(\L, \M )\to C^{n+1}(\L,\M)$ is thus given by
\begin{align}\nonumber(\delta f)_{\l_1,\cdots,\l_{n}}&(p_1,\cdots, p_{n+1})\\&=\nonumber \sum_{i=1}^{n+1} (-1)^{i+1}\rho(p_i)_{\l_i}f_{\l_1,\cdots,\hat{\l_{i}},\cdots, \l_{n}}(p_1,\cdots,\hat{p_i},\cdots, p_{n+1})\\&\label{coboundarymap}+ \sum_{i<j}(-1)^{i+j} f_{\l_i+ \l_j ,\l_1,\cdots, \hat{\l_{i}},\cdots, \hat{\l_{j}},\cdots, \l_{n}}([{p_{i}}_{\l_i} p_j], p_1, \cdots, \hat{p_i},\cdots, \hat{p_j},\cdots,p_{n+1}), \end{align} for $f\in C^{n}(\L, \M)$ and $p_1, p_2, \cdots, p_n \in \L$, where $\hat{p_i}$ shows the omision of the term $p_i$. It is worth noting that $\delta^2=0$. The pair of the graded space of cochains $C^*(\L,\M)=\oplus_{n=0}^{\infty} C^n(\L,\M)$ and coboundary operator $\delta$, i.e, $ \{C^*(\L,\M),\delta\}$ is called cochain complex. The space of cohomology group corresponding to this cochain complex is called the cohomology of the Lie conformal algebra $(\L,[\cdot_\l\cdot])$ with the coefficients in the representation $(\M,\rho)$ and is denoted by $H^*(\L,\M)$.
\par Next, we define the Nijenhuis-Richardson bracket with degree $-1$ graded Lie bracket
\begin{equation*}
[\cdot,\cdot]_{NR}: C^{m}(\L,\L) \times C^{n}(\L,\L) \to C^{m+n-1}(\L,\L),
\end{equation*}\begin{equation*}
 [J,K]_{NR}:= J \circledcirc K+ (-1)^{(m-1)(n-1)} K\circledcirc J,
\end{equation*}
for $m, n \geq 1$, where
\begin{equation}
\begin{aligned}
(J\circledcirc K)_{\l_1,\cdots,\l_{n+m-2}}(p_1,\cdots,p_{m+n-1})=\sum_{\tau\in S_{n,m-1}} &(-1)^{\tau} J_{\l_{\tau(1)},\cdots,\l_{\tau(n+m-2)}} (K_{\l_{\tau(1)},\cdots,\l_{\tau(n-1)}}(p_1, \cdots, p_n),\\& p_{\tau(n+1)}, \cdots, p_{\tau(n+m-1)}).
\end{aligned}
\end{equation} where $S_{n,m-1}$ denotes $\{1,2,\cdots,n+m-1\}$-shuffle in the permutation group of $S_{n+m-1}$ and $|\tau|$ denotes the sign of the permutation $\tau$. Note that, an element $m_c\in C^2(\L,\L)$ induces a skew-symmetric $\l$-bracket $[\cdot_\l\cdot]$ on $\L$ via $m_c(p,q)= [p_\l q]$, for $p,q \in \L$. The pair $(\L,[\cdot_\l\cdot])$ is called as Lie conformal algebra if and only if $[m_c, m_c]_{NR}= 0$. More specifically,
\begin{equation*}
{[m_c, m_c]}_{NR}= m_c \circledcirc m_c+ m_c\circledcirc m_c= 2(m_c \circledcirc m_c)
\end{equation*}This implies that, \begin{equation*}
m_c \circledcirc m_c=\frac{1}{2}{[m_c,m_c]}_{NR}
\end{equation*} The above identity is useful for further calculations in this paper. If we choose second $J=m_c\in C^2(\L,\L)$ and $K=f\in C^{n}(\L,\L)$, we observe that $[m_c, f]_{NR}=(-1)^{n-1}\delta (f).$ Moreover, we can say $diff={[m_c, f]}_{NR}$ is a differential of the Nijenhuis-Richardson graded Lie algebra $(C^*(\L,\L), [\cdot,\cdot]_{NR})$. The corresponding differential graded Lie algebra is denoted by $(C^*(\L,\L), {[\cdot,\cdot]}_{NR}, 
 diff=[m_c, f]_{NR})$.
\par Next, we construct a graded Lie algebra of the Lie conformal algebra, whose Maurer-Cartan elements are characterized by the Nijenhuis operators. This description allows us to connect the cohomology with any Nijenhuis operator. Let $(\L, [ \cdot_\l\cdot ])$ be a Lie conformal algebra with the graded space of cochains $C^* (\L,\L) =\bigoplus_{n\geq1} C^n (\L,\L)$. The graded space $C^* (\L,\L)$ possesses an associative product given by the cup operation, as shown below 
\begin{eqnarray}
 \begin{aligned}
(J\cup K)_{\l_1,\cdots,\l_{m+n-1}}(p_1,\cdots,p_{m+n}) &= \sum_{\tau\in S_{m,n}}(-1)^\tau J_{\l_1,\cdots,\l_{m-1}}(p_{\tau(1)},\cdots ,p_{\tau(m)})_{\l_1+\cdots+\l_{m}}\\&
K_{\l_{m+1},\cdots,\l_{m+n-1}}(p_{\tau(m+1)},\cdots ,p_{\tau(m+n)}),
\end{aligned}
\end{eqnarray}for $J\in C^m(\L,\L)$ and $K\in C^n(\L,\L)$. This cup product makes the pair $(C^* (\L,\L), \cup)$ into a graded Lie algebra. Further, note that the map
\[\circledcirc:
C^m(\L,\L) \times C^n(\L,\L) \to C^{m+n-1}(\L,\L), \quad (J, K) \mapsto J \circledcirc K
\]
defines an action of the Nijenhuis-Richardson graded Lie algebra on the cup-product graded Lie algebra. As a result, one shows that the bracket
\begin{equation}
[J, K]_{FN} := J \cup K + (-1)^m {\delta (J)}\circledcirc K - (-1)^{(m+1)n} {\delta (K)}\circledcirc J \label{eq:FNbracket} 
\end{equation}
makes the graded space $C^* (\L,\L)$ into a graded Lie algebra. The bracket defined in \eqref{eq:FNbracket} is called the Fr\"olicher- Nijenhuis bracket and the graded Lie algebra $(C^* (\L,\L), {[\cdot, \cdot]}_{FN})$ is called the Fr\"olicher- Nijenhuis algebra associated to the Lie conformal algebra $(\L, [\cdot_\l \cdot])$. Moreover, for any $J \in C^n (\L,\L) $ and $K \in C^m (\L,\L)$, we have (\cite{ADas})
$$\delta ({[J,K]}_{FN}) = [\delta (J), \delta (K)]_{NR}.$$
 When we consider $J \in C^1 (\L,\L) $ and $K \in C^1 (\L,\L)$ , then for any $p, q\in\L$ we have
 $${[J, K]}_{FN}(p, q) = [{J(p)}_\l K(q)]+[{K(p)}_\l J(q)]+ (JK+ KJ)([p_\l q])- K([J(p)_\l q]+ [p_\l J(q)])- J([K(p)_\l q]+[p_\l K(q)]).$$ 
 It implies a $\mathbb{C}[\p]$-linear map $\N:\L\to\L$ on the Lie conformal algebra $(\L,[\cdot_\l\cdot])$ is a Nijenhuis operator iff 
$${[\N, \N]}_{FN}(p, q)=0.$$ It also implies that $\N$ is a Maurer-Cartan element of the Fr\"olicher-Nijenhuis algebra. 
The preceding discussion yields the following Theorem:
\begin{thm}\label{thmFN}Assume that $(\L,[\cdot_\l\cdot] )$ be a Lie conformal algebra. Then the Fr\"olicher- Nijenhuis bracket ${[-,-]}_{FN}$ along with the graded space $C^*(\L,\L)$ forms a graded Lie algebra. A Nijenhuis operator on the Lie conformal algebra $(\L,[\cdot_\l\cdot])$ is a $\mathbb{C}[\p]$-linear map $\N \in C^1(\L,\L)$, if and only if $\N$ serves as a Maurer-Cartan element in the graded Lie algebra $(C^*(\L,\L), {[-,-]}_{FN})$, i.e., $ {[\N,\N]}_{FN}=0.$
\end{thm}Moreover, for $J \in C^0 (\L, \L) $ and $K \in C^1 (\L, \L)$, we get \begin{align}\label{eqeq}
 {[p, K]}_{FN}(q)=[{p}_\l K(q)]- K([p_\l q]).
\end{align}
\subsection{ Cohomology of Nijenhuis operators on a Lie conformal algebra}
 Let $(\L, [\cdot_\l \cdot],\N)$ be a Nijenhuis Lie conformal algebra.
 According to the preceding Theorem \ref{thmFN}, a Nijenhuis operator $\N$ induces a differential, given by
 $$d_{\N} = {[\N,-]}_{FN}: C^{*}(\L,\L) \to C^{*+1}(\L,\L).$$
Thus, cohomology of the cochain complex $ \{C^{*}(\L,\L), d_{\N} \} $ is called the cohomology of the Nijenhuis operator $\N$. Note that the differential $d_{\N}$ makes the triple $(C^{*}(\L,\L),[-,-]_{FN}, d_{\N})$ a differential graded Lie algebra.
This differential graded Lie algebra controls the linear deformation of the Nijenhuis operator. Consider the following theorem:
\begin{thm}
Consider a Lie conformal algebra $(\L,[\cdot_\l\cdot])$ and let $\N$ be its Nijenhuis operator. Then the sum $\N + \N'$ is a Nijenhuis operator on $(\L,[\cdot_\l\cdot])$ if and only if for $\N' \in C^{1}(\L,\L)$, $\N'$ is a Maurer-Cartan element in the differential graded Lie algebra $(C^{*}(\L,\L), [-,-]_{FN}, d_{\N})$. More precisely;
	\begin{equation*} d_{\N}(\N') + \frac{1}{2} [\N',\N']= 0 \implies [\N +\N',\N+\N']_{FN}= 0. \end{equation*}
 \end{thm}
\begin{proof}Here we give one line proof as follows
\begin{align*}
\relax {[\N +\N', \N+\N']}_{FN}= {[\N, \N']}_{FN}+{[\N', \N]}_{FN}+{[\N', \N']}_{FN}=2(d_{\N}(\N')+\frac{1}{2}[\N',\N'])=0.
 \end{align*}
\end{proof} 
 Since $\N$ is a Maurer-Cartan element that induces a differential, it is found that $d_{\N}^2=0$. Explicitly, the map $d_\N$ is given by 
\begin{equation}\begin{aligned}\label{coboundaryNij}&(d_{\N} f)_{\l_1,\cdots,\l_{n}}(p_1,\cdots, p_{n+1})\\=& 
 \sum_{i=1}^{n+1} (-1)^{i+1} [(\N p_i)_{\l_i}f_{\l_1,\cdots,\hat{\l_{i}},\cdots, \l_{n}}(p_1,\cdots,\hat{p_i},\cdots, p_{n+1})]\\&
+ \sum_{1\leq i<j\leq n+1}(-1)^{i+j} f_{\l_i+ \l_j ,\l_1,\cdots, \hat{\l_{i}},\cdots, \hat{\l_{j}},\cdots, \l_{n}}([{\N (p_{i})}_{\l_i} p_j]+[{p_{i}}_{\l_i} \N (p_j)]-\N ([{p_{i}}_{\l_i} p_j]), p_1, \cdots, \hat{p_i},\cdots, \hat{p_j},\cdots,p_{n+1})
\\&-\N( \sum_{i=1}^{n+1} (-1)^{i+1} [{p_i }_{\l_i}f_{\l_1,\cdots,\hat{\l_{i}},\cdots, \l_{n}}(p_1,\cdots,\hat{p_i},\cdots, p_{n+1})]\\& + \sum_{1\leq i<j\leq n+1}(-1)^{i+j} f_{\l_i+ \l_j ,\l_1,\cdots, \hat{\l_{i}},\cdots, \hat{\l_{j}},\cdots, \l_{n}}([ {p_i} _{\l_i} p_j], p_1, \cdots, \hat{p_i},\cdots, \hat{p_j},\cdots,p_{n+1}) ),\end{aligned}\end{equation}
where $f\in C^{n}_{\N}(\L,\L)$ and $p_1, p_2,\cdots, p_{n+1}\in \L$.
The cohomology of this cochain complex is denoted by \textbf{$H^{*}_{\N}(\L,\L)$ or simply by $ H^{*}_{\N}(\L)$}. From the above Eqs. \eqref{eqeq}, and \eqref{coboundaryNij}, we see that the coboundary map of $d_\N$ cannot be expressed as the coboundary map $\delta$ of the induced Lie conformal algebra $(\L,{[\cdot_\l\cdot]}_\N)$ with coefficients in a suitable representation. Specifically, the cohomology of the Nijenhuis operator is different from the Cohomology of the induced Lie conformal algebra with the coefficients in the adjoint representation. However, in the next result, we show that there is a homomorphism between the cohomology of the Nijenhuis operator $\N$ and the induced Lie conformal algebra $(\L,{[\cdot_\l\cdot]}_\N)$. For $n\geq 0$, we define $\varPhi^n: C^
n(\L,\L) \to C^{n+1}(\L,\L)$ by $\varPhi^n(f):=(-1)^{n+1}\delta (f),$ for any $f\in C^
n(\L,\L).$ Where $\delta$ is the coboundary operator of the Lie conformal algebra with coefficients in the adjoint representation. 
\begin{prop}\label{propg}
 The collection of maps $\{\varPhi_n\}_{n=0}^\infty$ satisfy $\delta_\N \circ \varPhi^n = \varPhi^{n+1} \circ d_\N, \quad \text{for all } n,$ where $\delta_\N$ is the coboundary operator of the induced Lie conformal algebra $(\L, {[\cdot_\l \cdot]}_\N)$ with coefficients in the adjoint representation. As a consequence, there is a homomorphism from the cohomology of the Nijenhuis operator $\N$, i.e., $H^*_\N(\L)$, to the cohomology of the induced Lie conformal algebra $H^*((\L,{[\cdot_\l\cdot]}_\N),(\L,{[\cdot_\l\cdot]}_\N))$.
\end{prop}
\begin{proof}
For any $f \in C^n(\L,\L)$, we have
\begin{align*}
(\varPhi^{n+1} \circ d_\N)(f) &= \varPhi^{n+1}({[\N, f]}_{FN}) \\
&= (-1)^{n+2} \delta({[\N, f]}_{FN}) \\
&= (-1)^{n+2} {[\delta( \N), \delta (f)]}_{NR}\\
&= -{[m_c, (-1)^{n+1} \delta f]}_{NR} \\
&= -{[m_c, \varPhi^n(f)]}_{NR} \\
&= (\delta_\N \circ \varPhi^n)(f).
\end{align*}
Here, $m_c = \delta (\N )\in C^2(\L,\L)$ is the element corresponding to the deformed Lie conformal bracket ${[\cdot_\l \cdot]}_\N$, and $\delta_\N$ is the coboundary operator of the deformed Lie conformal algebra with coefficients in the adjoint representation.
\end{proof} 
\subsection{ Formal deformations of Nijenhuis operators on Lie conformal algebra}
In the following, we examine one-parameter formal deformations of Nijenhuis operators on Lie conformal algebras. We take into account finite order deformations as well as their extensions to deformations of the subsequent order. Consider a Lie conformal algebra $(\L,[\cdot_\l\cdot])$, a Nijenhuis operator $\N : \L\to \L$ and the formal power series space $\L \llbracket t \rrbracket $ with coefficients from $\L$. We observe that the bilinear $\l$-bracket $[\cdot_\l\cdot]$ can be extended to $\L \llbracket t \rrbracket $ by the $\mathbb C \llbracket t \rrbracket $-linearity. We use the same notation to represent extended structures. These extensions result in a Lie conformal algebra over $\mathbb C \llbracket t \rrbracket $ denoted by $(\L \llbracket t \rrbracket , [\cdot_\l\cdot])$.
\begin{defn}
	A one-parameter formal deformation of Nijenhuis operator $\N$ consists of a formal sum $$\N_t=\N_{0}+\N_{1}t+\N_{2}t^2+\cdots\in C^1(\L,\L) \llbracket t \rrbracket $$ having $\N_{0} =\N$ in such a way that the $\mathbb C \llbracket t \rrbracket $-linear map $\N_{t} :\L \llbracket t \rrbracket \to \L \llbracket t \rrbracket $ is a Nijenhuis operator on the Lie conformal algebra $(\L \llbracket t \rrbracket , [\cdot_\l\cdot]).$
\end{defn}Therefore, it follows from the Definition \ref{Nijenhuis}, $\N_{t}$ satisfies  
\begin{equation*}
[{\N_{t}(p)}_{\l} \N_{t}(q)]= \N_{t}\Big([{\N_{t}(p)}_{\l} q]+ [{p}_{\l}\N_{t}(q)]- \N_{t}([p_{\l} q])\Big),\quad \text{ for } p, q\in \L .\end{equation*}
Equivalently, we have 
\begin{equation*} \sum_{i+j=n}[{\N_{i}(p)}_{\l} \N_{j}(q)]= \sum_{i+j=n}\N_{i}([{\N_{j}(p)}_{\l} q]+ [p_{\l}\N_{j}(q)]- \N_{j}([p_{\l} q])), \quad \text{ for } n = 0, 1, 2, \cdots.\end{equation*}
This condition is satisfied automatically for $n=0$, since $\N_{0} = \N$ is a Nijenhuis operator on the Lie conformal algebra $(\L,[\cdot_\l\cdot])$. For $n=1$, we have
\begin{equation}\label{eqteri}
\begin{aligned}
\relax [{\N_{1}(p)}_{\l} \N(q)]+[{\N(p)}_{\l} \N_{1}(q)]=& \N_{1}([{\N(p)}_{\l} q]+ [p_{\l}\N(q)]- \N([p_{\l} q]))\\&+\N([{\N_1(p)}_{\l} q]+ [p_{\l}\N_{1}(q)]- \N_{1}([p_{\l} q])).
\end{aligned}
\end{equation}
Hence Eq. \eqref{eqteri} demonstrates that $\N_{1}$ is a $1$-cocycle in the cohomology of the Nijenhuis operator $\N$. This is referred to as an infinitesimal deformation. In general, $\N_{k}$ is a $1$-cocycle in the cohomology of $\N$ if $\N_{1} = \cdots= \N_{k-1} = 0$ and $\N_{k}$ is the first nonzero term.
\begin{defn}
	Two deformations of a Nijenhuis operator $\N$ with the notion $\N_{t} =\sum_{i=1}^{\infty}t^{i}\N_{i}$ and $\N_{t}' =\sum_{i=1}^{\infty} t^{i}\N_{i}'$ are equivalent if there exist linear maps $\psi_i \in C^1(\L,\L)$, for $i\geq2$ such that $$\psi_t:= (id + t[{p}_\l-]+ \sum_{i\geq2}t^i\psi_i): \L \llbracket t \rrbracket \to \L \llbracket t \rrbracket $$ is a morphism of Nijenhuis operators from $\N_{t}$ to $\N_{t}'$.
\end{defn}
The morphism $\psi_t: \L\llbracket t \rrbracket\to\L\llbracket t\rrbracket$  satisfies the condition $\N_{t}' \circ \psi_t = \psi_t \circ \N_{t}$. For any $q$ in $\L$ we have
\begin{equation*}
(\N + t\N_{1}')(id + t{([{p}_\l-])}+ \sum_{i\geq2}t^i\psi_i)(q)= (id+ t{([p_\l-])} 
+\sum_{i\geq2}t^{ i}\psi_i) (\N + t\N_{1})(q) ~~\text( mod~ t^2 ).
\end{equation*}By equating coefficients of $t$, we get $$\N_{1}(q) -\N'_{1}(q) = \N([p_\l q]) -[p_\l \N(q)]= {[\N, p]}_{FN}(q)= d_{\N}(p)(q).$$ Thus, we have the following theorem.
 \begin{thm}
 Consider a one-parameter formal deformation of $\N$ as $\N_{t}=\sum_{i=0}^{\infty} t^{i} \N_{i}$. As a result, the linear term $\N_{1}$ is a $1$-cocycle in the cohomology of the Nijenhuis operator $\N$, whose cohomology class solely depends on the equivalence class of the deformation. 
 \end{thm}
\subsection{ Extensions of the deformation of finite order} Consider a Lie conformal algebra $(\L,[\cdot_\l\cdot])$. The Lie bracket $[\cdot_\l\cdot]$ on $\L$ can be extended to $\L \llbracket t \rrbracket /t^{k+1}$ by $\mathbb C \llbracket t \rrbracket /t^{k+1}$-linearity making $(\L \llbracket t \rrbracket / t^{k+1}, [\cdot_\l\cdot] )$ a Lie conformal algebra over $\mathbb{C} \llbracket t \rrbracket / t^{k+1}$ modules. %As a result, the space $\L \llbracket t \rrbracket / t^{k+1}$ is a $\mathbb{C} \llbracket t \rrbracket / t^{k+1}$-module. 
\begin{defn}Consider a Lie conformal algebra $(\L, [\cdot_{\l} \cdot])$ and $\N: \L \to \L$ be a Nijenhuis operator on it. An order $k$ deformation of the Nijenhuis operator $\N$ comprises a sum $$\N_{t} = \sum_{i=1}^{k}\N_{i}t^i \in C^1(\L,\L) \llbracket t \rrbracket \ t^{k+1},$$ where $\N_{t}$ is a Nijenhuis operator on the Lie conformal algebra $(\L \llbracket t \rrbracket / t^{k+1}, [\cdot_\l\cdot])$ over $\mathbb{C} \llbracket t \rrbracket /t^{k+1}$ and $\N_{0}= \N$.\end{defn}
The following condition holds for the deformation of order $n$, for all $p, q\in \L$ and $n= 0, 1,\cdots, k$. 
\begin{equation}\label{eqmeri}
\sum_{i+j=n}[\N_{i}(p)_{\l} \N_{j}(q)]= \sum_{i+j=n}\N_{i}([\N_{j}(p)_{\l} q]+ [p_{\l}\N_{j}(q)]- \N_{j}([p_{\l} q])). \end{equation}
If $\N$ is extensible, then deformation of order $n+1$ needs to be satisfied, namely 
\begin{equation*}
\sum_{i+j=n+1}[\N_{i}(p)_{\l} \N_{j}(q)]= \sum_{i+j=n+1} \N_{i}([\N_{j}(p)_{\l} q]+ [p_{\l}\N_{j}(q)]- \N_{j}([p_{\l} q])).\end{equation*} 
Equivalently

%\begin{equation*}\begin{aligned}&[\N(p)_{\l} \N_{n+1}(q)]+[\N_{n+1}(p)_{\l} \N(q)]+\sum_{i+j=n+1}[\N_{i}(p)_{\l} \N_{j}(q)]\\& =\N([\N_{n+1}(p)_{\l} q]+ [p_{\l}\N_{n+1}(q)]- \N_{n+1}[p_{\l} q])\\&+\N_{n+1}([\N(p)_{\l} q]+ [p_{\l}\N(q)]- \N[p_{\l} q])\\&+\sum_{i+j=n+1} \N_{i}([\N_{j}(p)_{\l} q]+ [p_{\l}\N_{j}(q)]- \N_{j}[p_{\l} q]).\end{aligned}\end{equation*}\begin{equation*}\begin{aligned}&\sum_{i+j=n+1}([\N_{g_i}(p)_{\l} \N_{g_j}(q)]- \N_{g_i}([\N_{g_j}(p)_{\l} q]+ [p_{\l}\N_{g_j}(q)]- \N_{g_j}[p_{\l} q]))\\=&-[\N_{g}(p)_{\l} \N_{g_{n+1}}(q)]-[\N_{g_{n+1}}(p)_{\l} \N_{g}(q)]+ \N_{g}([\N_{g_{n+1}}(p)_{\l} q]+ [p_{\l}\N_{g_{n+1}}(q)]- \N_{g_{n+1}}[p_{\l} q])\\&+\N_{g_{n+1}}([\N_{g}(p)_{\l} q]+ [p_{\l}\N_{g}(q)]- \N_{g}[p_{\l} q])\\=&d(\N_{g_{n+1}})(p,q), \quad \forall p, q\in \L.\end{aligned}\end{equation*}

\begin{equation*}
\begin{aligned}
&\sum_{i+j=n+1}([\N_{i}(p)_{\l} \N_{j}(q)]- \N_{i}([\N_{j}(p)_{\l} q]+ [p_{\l}\N_{j}(q)]- \N_{j}[p_{\l} q]))=d_{\N}(\N_{n+1})(p,q).
\end{aligned}\end{equation*}
That can also be written as
 \begin{equation}\label{aiza}
d_{\N }(\N_{ n+1} )= -\frac{1}{2}\sum_{i+j=n+1,i,j>1}[\N_{i} , \N_{j}]_{FN}.
\end{equation}
Note that the right-hand side of the Eq. \eqref{aiza} does not contain $\N_{n+1}$, therefore it depends only on the order $n$ deformation $N_t$. It is known as the obstruction to extending the deformation $\N_{t}$. We denote the obstruction by $Ob_{\N_{t}}$.
%\begin{defn}A deformation $N_t =\sum_{i=0}^{k}t^iN_i$ of order $k$ is said to be extensible if there exists an element $N_{k+1} \in Hom(\L,\L)$ with $\a \circ N_{k+1}= N_{k+1} \circ \a$ such that $\bar{N_t }= \sum_{i=0}^{k+1}t^iN_i$ defines a deformation of order $k+1$. \end{defn}In other words, one is looking for an element $N_{k+1} \in Hom(\L,\L)$ with $\a \circ N_{k+1} = N_{k+1} \circ \a$ that satisfies \begin{equation}\label{eqhum}d_{N}(N_{k+1})= -\frac{1}{2}\sum_{i+j=k+1,i,j>1}[{N_i} _\l N_j]_{FN},\end{equation}Note that the right hand side of Eq.\ref{eqhum} has no $N_{k+1}$ and therefore it depends only on the order $k$ deformation $N_t$. This is called the obstruction to extend the deformation and is denoted by $Ob_{N_t}$.

\begin{prop}\label{eqtum}
	In the cohomology of $\N$, the obstruction is a $2$-cocycle, i.e., $d_{\N} (Ob_{\N_{t}} )=0$.
\end{prop}
\begin{proof}We have\begin{equation*}\begin{aligned}d_{\N}(Ob_{\N_{t}})&= -\frac{1}{2}\sum_{i+j=n+1,i,j>1}{[\N, {[\N_{i} , \N_{j}]}_{FN}]}_{FN}\\&= -\frac{1}{2}\sum_{i+j=n+1,i,j>1}( {[{[ \N, \N_{i}]}_{FN}, \N_{j}]}_{FN} -{[\N_{i},{[\N, \N_{j}]}_{FN}]}_{FN})\\&=\frac{1}{4}\sum_{i_1+i_2+j= n+1, i_1,i_2,j>1} {[{[\N_{{i}_1}, \N_{{i}_{2}}]}_{FN}, \N_{j}]}_{FN}- \frac{1}{4}\sum_{i+j_1+j_2= n+ 1,i,j_1,j_2>1}{[\N_{i}, {[\N_{{j}_1}, \N_{{j}_2}]}_{FN}]}_{FN}\\&= \frac{1}{2}\sum_{i+j= n+ 1,i,j>1} {[{[ \N_{i},\N_{j}]}_{FN}, \N_{k}]}_{FN}\\&=0.
\end{aligned}\end{equation*}
\end{proof}Thus, from the Proposition \ref{eqtum} and Eq. \eqref{aiza}, the following theorem applies.\begin{thm}
Let $\N_{t}$ be a deformation of a Nijenhuis operator $\N$ of order $n$. This deformation $\N_{t}$ on Lie conformal algebra is extensible if and only if the obstruction class $[Ob_{\N_{t}}]\in {H}^2_{\N}(\L,\L)$ is trivial, where we denote ${H}^2_{\N}(\L,\L)$ as a cohomology class associated with the $2$-cochain complex.
\end{thm}
\begin{cor}
Any finite order deformation of $\N$ is extensible, if ${H}^2_{\N}(\L,\L)= 0$.
\end{cor}
\subsection{ Representation and Cohomology of Nijenhuis Lie conformal algebra}
Let $(\L,[\cdot_\l\cdot])$ be a Lie conformal algebra, and $\N$ be a Nijenhuis operator on it, we call it a Nijenhuis Lie conformal algebra, denoted by $(\L, [\cdot_\l\cdot ], \N)$. The  homomorphism between two Nijenhuis Lie conformal algebras $(\L,[\cdot_\l\cdot], \N)$ and $(\L',[\cdot_\l\cdot ]', \N')$, is a Lie conformal algebra homomorphism, satisfying $\psi\circ \N'= \N \circ \psi$. If $\psi$ is bijective, we call it an isomorphism. Next, we define the representation of a Nijenhuis Lie conformal algebra.
\begin{defn}\label{nijrep}
A representation of a Nijenhuis Lie conformal algebra $(\L,[\cdot_\l\cdot ], \N)$ consists of the triple $(\M,\rho,\N_\M)$, where $(\M,\rho)$ is a conformal representation of the Lie conformal algebra and $\N_\M\in Cend(\M)$ is a $\mathbb{C}[\partial]$-linear map, such that following identity holds: 
\begin{align}\label{eqrep}
{\rho(\N(p))}_\l \N_\M(m)=\N_\M\Big({\rho(\N (p))}_\l m+ {\rho(p)}_\l \N_\M(m) -\N_\M ({\rho(p)}_\l m) \Big), \end{align} for all $p\in \L$, $m \in \M$. We denote the representation of $ (\L, [\cdot_\l\cdot ], \N)$ by $ (\M,\rho,\N_\M)$. 
\end{defn}
 \begin{ex}
 Consider a Nijenhuis Lie conformal algebra $ (\L,[\cdot_\l\cdot ], \N)$, then: 
 \begin{enumerate}
 \item The triple $(\L, ad_\L,\N)$ is a Nijenhuis representation on it, where $ad_\L:\L\to\L$ is called adjoint representation, given by $ad_\L(p)=[p_\l -]$ for any $q\in \L.$
 \item The triple $(\M, \rho, 0)$ is a Nijenhuis representation on it.
 \item The triple $(\M, \rho, \pm id_\M)$ is a Nijenhuis representation on it.
 \end{enumerate}
 \end{ex}
 Similar to the Definition \ref{nijrep}, we can define Nijenhuis representation on the \textbf{induced Nijenhuis Lie conformal algebra} $(\L, {[\cdot _\l\cdot]}_{\N^k}, \N^l)$.
\begin{defn} \label{defrepex}Let $(\L,{[\cdot_\l\cdot]}_{\N^k}, \N^l)$ be a Nijenhuis Lie conformal algebra, then its representation is given by the triple $(\M, \rho^k, \N^l_\M)$, satisfying the following identity: 
\begin{align}\label{eqrep1}
\rho^k(p)_\l (m)= \rho(\N^k (p))_\l m+ \rho(p)_\l \N_\M^k(m) -\N_\M^k (\rho(p)_\l m) , \end{align} for all $p\in \L$, $m \in \M$. 
\end{defn}
\begin{prop}
 Given a Nijenhuis Lie conformal algebra $ (\L,[\cdot_\l\cdot ], \N)$ and its Nijenhuis representation $(\M,\rho,\N_\M)$, then $(\L \oplus \M, {[\cdot_\l\cdot ]}_{\rho}, \N+\N_\M)$ is also Nijenhuis Lie conformal algebra, denoted by $\L \ltimes \M$.
\end{prop}
\subsubsection{The cohomology of Nijenhuis Lie conformal algebra with the coefficients in the adjoint Nijenhuis representation}
 After presenting the cohomology of Lie conformal algebra and Nijenhuis operators, our next task is to provide the cohomology of Nijenhuis Lie conformal algebra $ (\L, [\cdot_\l\cdot ], \N)$ with coefficients in the adjoint Nijenhuis representation $ (\M,\rho=ad,\N_\M)$. To achieve this, we combine the cochain complex of Lie conformal algebra and the cochain complex of Nijenhuis operator. But we first provide the relation between the aforementioned cochain complexes $(C^*(\L,\L),\delta),$ and $(C^*_{\N}(\L,\L),d_{\N})$ by defining a homomorphism map $\xi$ as follows: 
 \begin{defn}Consider a Nijenhuis Lie conformal algebra $(\L, [\cdot_\l \cdot], \N)$. Define a map $\xi^n: C^{n}(\L,\L)\to C^{n}_{\N}(\L,\L)$ by \begin{align}&\nonumber \xi^nf_{\l_1,\l_2,\cdots,\l_{n-1}}(p_1,p_2,\cdots,p_n) \\&= f_{\l_1,\cdots,\l_{n-1}}(\N (p_1),\N (p_2),\cdots,\N (p_n)) - \sum_{i=1}^{n} (\N \circ f)_{\l_1, \cdots,\l_{n-1}}(\N (p_1), \cdots, p_i, \cdots, \N (p_n))\\&\nonumber +\sum_{1\leq i,j \leq n } (\N^2 \circ f)_{\l_1 ,\cdots,\l_{n-1}}(\N (p_1), \cdots, p_i, \cdots, p_j, \cdots, \N (p_n)) -\cdots+ (-1)^n (N^n \circ f)_{\l_1 ,\cdots,\l_{n-1}}(p_1,p_2,\cdots,p_n),\end{align} for all $p_1,\cdots,p_n\in \L$ and $f\in C^{n}(\L, \L).$ \end{defn}It is easy to observe that \begin{align}
 d_\N \circ \xi= \xi \circ \delta.
 \end{align}
 The cochain groups of Nijenhuis Lie conformal algebra is thus given by \begin{align*}
 C^0_{\N L}(\L, \L) &= C^0(\L, \L), \quad &\text{ for } n= 0\\
 C^1_{\N L}(\L, \L) &= C^1(\L, \L), \quad &\text{ for } n = 1\\
 C^n_{\N L}(\L, \L) &= C^n(\L,\L) \oplus C^{n-1}_{\N}(\L,\L), \quad & \forall~ n \geq 2 
 \end{align*} and the coboundary map $d^n_{\N L}: C^n_{\N L}(\L,\L) \to C^{n+1}_{\N L}(\L,\L)$ given by
\begin{align*}
 d^n_{\N L}(f) = (\delta^n(f),- \xi^n(f)),
\end{align*}
\begin{align*}
 d^n_{\N L}(f,g) = \Big(\delta^n(f),(-1)^n \xi^n(f)+ d_\N^{n-1}(g)\Big),
\end{align*} 
 for any $f \in C^n(\L,\L)$ and $g \in C^{n-1}_{\N}(\L,\L)$. 
\begin{thm}
The map $d^n_{\N L} : C^n_{\N L}(\L,\L) \to C^{n+1}_{\N L}(\L,\L)$ satisfies $d^{n+1}_{\N L} \circ d^n_{\N L} = 0$.
\end{thm}
\begin{proof}
 Let $f \in C^n(\L,\L)$ and $g \in C^{n-1}_{\N}(\L,\L)$, then we have
\begin{align*}
 d^{n+1}_{\N L} \circ d^n_{\N L}(f,g) &= d^{n+1}_{\N L}\Big(\delta^n(f),(-1)^n \xi^n(f)+ d_\N^{n-1}(g)\Big) \\&= (\delta^{n+1}(\delta^n(f)),(-1)^{n+1}\xi^{n+1}(\delta^n(f)) + d_\N^n((-1)^n \xi^n(f)+ d_\N^{n-1}(g)))\\&= (0, (-1)^{n+1}\xi^{n+1}(\delta^n(f)) + (-1)^n d_\N^n(\xi^n(f))+ d_\N^n d_\N^{n-1}(g)) \\&= (0, (-1)^{n+1}\xi^{n+1}(\delta^n(f)) + (-1)^n d_\N^n(\xi^n(f))) \\&= 0.
\end{align*}\end{proof}
Thus, $\{C^n_{\N L}(\L,\L), d^n_{\N L}\}$ is the cochain complex of Nijenhuis Lie conformal algebras. Its cohomology class is denoted by $H^*_{\N L}(\L,\L)$.
 All the cochain complexes we studied above can be expressed in the form of the following short exact sequence
\begin{align}
 0 \longrightarrow \{C^n_{\N}(\L,\L),d_\N\} \longrightarrow \{C^n_{\N L}(\L,\L),d_{\N L}\} \longrightarrow \{C^n (\L,\L),\delta\} \longrightarrow 0.
\end{align} 
\subsubsection{The cohomology of the Nijenhuis Lie conformal algebra, with the coefficients in the Nijenhuis representation}
In the following, we generalize the above study to present the cohomology of the Nijenhuis Lie conformal algebra, with the coefficients in the Nijenhuis representation. Assume that $ (\L, [\cdot_\l \cdot],\N)$ be a Nijenhuis Lie conformal algebra with representation $(\M ,\rho, \N_\M)$. 
There are various cochain complexes, and we will consider them with respect to representation one by one.
\begin{enumerate}
 \item For the cochain complex $\{C^{n}(\L, \M),\delta)\}$, the coboundary map $\delta$ is given by Eq. \eqref{coboundarymap}. 
 \item For the cochain complex $\{C^{*}_{\N}(\L, \M)=\oplus_{n\geq 1}C^{n}_{\N}(\L,\M),d_{\N,\M})\}$,\\ the coboundary map $d_{\N,\M}:C^{n}_{\N}(\L,\M)\to C^{n+1}_{\N}(\L,\M)$ can be constructed from Eq. \eqref{coboundaryNij}, such that
\begin{equation}\begin{aligned}\label{coboundaryNij1}(d_{\N,\M} f)_{\l_1,\cdots,\l_{n}}&(p_1,\cdots, p_{n+1})\\=& 
 \sum_{i=1}^{n+1} (-1)^{i+1} \rho( \N (p_i))_{\l_i}f_{\l_1,\cdots,\hat{\l_{i}},\cdots, \l_{n}}(p_1,\cdots,\hat{p_i},\cdots, p_{n+1}) \\&
+ \sum_{1\leq i<j\leq n+1}(-1)^{i+j} f_{\l_i+ \l_j , \l_1, \cdots, \hat{\l_{i}},\cdots, \hat{\l_{j}}, \cdots, \l_{n}}\\&\Big(([{\N(p_{i})}_{\l_i} {p_j}]+ [{p_{i}}_{\l_i} {\N(p_j)} ]- \N([{p_{i}}_{\l_i} p_j])), p_1, \cdots, \hat{p_i},\cdots, \hat{p_j},\cdots,p_{n+1}\Big)
\\&-\N_\M \Big(\sum_{i=1}^{n+1} (-1)^{i+1} \rho{(p_i)}_{\l_i}f_{\l_1,\cdots,\hat{\l_{i}},\cdots, \l_{n}}(p_1,\cdots,\hat{p_i},\cdots, p_{n+1}) \\& + \sum_{1\leq i<j\leq n+1}(-1)^{i+j} f_{\l_i+ \l_j ,\l_1,\cdots, \hat{\l_{i}},\cdots, \hat{\l_{j}},\cdots, \l_{n}}([{p_{i}}_{\l_i} {p_j}], p_1, \cdots, \hat{p_i},\cdots, \hat{p_j},\cdots,p_{n+1})\Big),\end{aligned}\end{equation}
where $f\in C^{n}_{\N}(\L,\M)$ and $p_1, p_2,\cdots, p_{n+1}\in \L$ and $m\in \M$. \\
For $n=0 $, Eq. \eqref{coboundaryNij1} becomes
\begin{equation}\begin{aligned}\label{coboundaryNij2}&d_{\N,\M} (m) (p) =\rho(\N (p))_{\l}m 
+ \N_\M(\rho{(p)}_{\l }m).\end{aligned}\end{equation}
 One verifies that $ d_{\N,\M}^n \circ d_{\N,\N}^{n+1} = 0 $, confirming that $ d_{\N,\M} $ defines a coboundary operator for the Nijenhuis Lie conformal algebra $ (\mathcal{L}, [\,\cdot_\lambda \cdot\,], \mathcal{N}) $ with coefficients in its Nijenhuis representation $(\M ,\rho, \N_\M)$. The associated cohomology is denoted by $ H^*_{\N,\M}(\L, \M) $.
\end{enumerate} 
\begin{rem} When Nijenhuis representation is adjoint Nijenhuis representation, i.e., $(\M,\rho,\N_\M)= (\L,ad_\L,\N)$, then $d_\N$ coincides with $d_{\N,\M}$. Hence, their cochain complexes.
\end{rem} 
For $n\geq 0$, we define $\varPhi^n: C^
n(\L,\M) \to C^{n+1}(\L,\M)$ by $\varPhi^n(f):=(-1)^{n+1}\delta (f)$, for $f\in C^n(\L,\M).$
Thus, similar to Proposition \ref {propg}, we have $\delta_{\N,\M} \circ \varPhi^n = \varPhi^{n+1} \circ d_{\N,\M}, \quad \text{for all } n.$ where $\delta_{\N,\M}$ is coboundary operator of the induced Lie conformal algebra with representaion $\M_\N$. Hence, there is a morphism between their cohomology groups.
 \begin{defn}Consider a Nijenhuis Lie conformal algebra $(\L, [\cdot_\l \cdot], \N)$. Define a map $\xi_{\N,\N_\M}^n: C^{n}(\L,\M)\to C^{n}_{\N,\M}(\L,\M)$ by \begin{align*}\xi^n_{\N,\N_\M}f_{\l_1,\l_2,\cdots,\l_{n-1}}(p_1,p_2,\cdots,p_n) &= f_{\l_1,\l_2,\cdots,\l_{n-1}}(\N (p_1),\N (p_2),\cdots,\N (p_n))\\&- \sum_{i=1}^{n} (\N_\M \circ f)_{\l_1,\l_2,\cdots,\l_{n-1}}(\N (p_1), \cdots, p_i, \cdots, \N (p_n))\\&+\sum_{1\leq i,j \leq n } (\N_\M^2 \circ f)_{\l_1,\l_2,\cdots,\l_{n-1}}(\N (p_1), \cdots, p_i, \cdots, p_j, \cdots, \N (p_n))\\&-\cdots+ (-1)^n (\N_\N^n \circ f)_{\l_1,\l_2,\cdots,\l_{n-1}}(p_1,p_2,\cdots,p_n),\end{align*} for all $p_1,\cdots,p_n\in \L$ and $f\in C^{n}(\L, \M).$ \end{defn}It is easy to observe that \begin{align}
 d_{\N,\M }\circ \xi_{\N,\N_\M}= \xi_{\N,\N_\M} \circ \delta.
 \end{align}Now, to define the cochain complex of Nijenhuis Lie conformal algebras with the coefficients in the Nijenhuis representation, we define the cochain groups by
 \begin{align*}
 C^0_{\N L}(\L, \M) &= C^0(\L, \M), \quad &\text{ for } n= 0\\
 C^1_{\N L}(\L, \M) &= C^1(\L, \M), \quad &\text{ for } n = 1\\
 C^n_{\N L}(\L, \M) &= C^n(\L,\M) \oplus C^{n-1}_{\N}(\L,\M), \quad & \forall~ n \geq 2 
 \end{align*} and the coboundary map $d^n_{\N L}: C^n_{\N L}(\L , \M ) \to C^{n+1}_{\N L}(\L, \M)$ given by
\begin{align*}
d^n_{\N L}(f) = (\delta^n(f),- \xi_{\N,\M}^n(f)), \quad \text{for n=1} 
\end{align*}
\begin{align*} d^n_{\N L}(f,g) = (\delta^n(f),(-1)^n \xi^n_{\N,\M}(f)+ d_{\N,\M}^{n-1}(g)),
\end{align*} for any $f \in C^n(\L,\M)$ and $g \in C^{n-1}_{\N}(\L,\M)$. 
\begin{thm}
The map $d^n_{\N L} : C^n_{\N L}(\L,\M) \to C^{n+1}_{\N L}(\L,\M)$ satisfies $d^{n+1}_{\N L} \circ d^n_{\N L} = 0$.
\end{thm} Thus, $\{C^*_{\N L}(\L, \M ),d^n_{\N L}\}$ is the cochain complex of Nijenhuis Lie conformal algebras with coefficients in the conformal representation. Its cohomology class is denoted by $H^*_{\N L}(\L,\M)$.
All the cochain complexes we studied in this subsection can be expressed in terms of following exact sequence
\begin{align}
 0 \longrightarrow \{C^n_{\N,\M}(\L,\M),d_{\N,\M}\} \longrightarrow \{C^n_{\N L}(\L,\M),d_{\N L}\} \longrightarrow \{C^n (\L,\M),\delta\} \longrightarrow 0.
\end{align} 
\section{ Skeletal and Strict $2$-terms Homotopy Nijenhuis Lie conformal algebra}   The homotopy version of conformal algebras, or equivalently, the conformal version of homotopy algebras, is a relatively recent area of research. So far, only a limited number of studies have been explored in this direction. For instance, the homotopy theory of associative conformal algebras is explored in \cite{HZ, SD}. This concept for $\L_{\infty}$ algebras is introduced in \cite{AWaveraging, KS, Z}, and the homotopy theory of Leibniz conformal algebra has been studied in \cite{DS}. Motivated by these developments, this section is devoted to introducing homotopy Nijenhuis operators on $2$-term $\L_\infty$-conformal algebras. A $2$-term $\L_\infty$-conformal algebra equipped with a homotopy Nijenhuis operator is referred to as a $2$-term Nijenhuis $\L_\infty$-conformal algebra. We pay particular attention to two special classes of these structures: the skeletal and the strict cases. Specifically, we show that skeletal $2$-term Nijenhuis $\L_\infty$-conformal algebras correspond to $3$-cocycles of Nijenhuis Lie conformal algebras. Further, we define crossed modules of Nijenhuis Lie conformal algebras and show that they are in one-to-one correspondence with strict $2$-term Nijenhuis $\L_\infty$-conformal algebras.
\begin{defn}
 An \textbf{$\L_\infty$-conformal algebras} are graded $\mathbb{C}[\p]$-modules $\L=\oplus_i\L_i$ where $i\in \mathbb{Z}$, equipped with the collection of graded $\mathbb C$-linear maps $\{l_k:\L^{\otimes k}\to \L[\l_1,\l_2,\cdots, \l_{k-1}]\}$ with degree $ k-2$ for any $k\in \mathbb{N}$, satisfying the following set of equations in the sense that
 \begin{enumerate}
 \item $l_k$ is \textbf{conformal sesquilinear}, i.e.,
 \begin{equation}\label{eq10}
	\begin{aligned} {l_k}_{\l_1, \l_2,\cdots, \l_{k-1}}&(p_1,p_2,\cdots,\p(p_i),\cdots,p_k)\\&=\begin{cases}
	-\l_{i} {l_{k}}_{\l_1,\l_2,\cdots,\l_{k-1}}(p_1, p_2, \cdots , p_k),&i= 1, \cdots, k-1, \\
	 (\p+\l_1+\l_2+\cdots+\l_{k-1}) {l_k}_{\l_1,\l_2,\cdots,\l_{k-1}}(p_1, p_2, \cdots , p_k),&i=k. \end{cases}
	\end{aligned}\end{equation}
 \item $l_k$ is \textbf{conformal skew-symmetric}, i.e.,
 \begin{align*}
{l_k}_{\l_1+\l_2+\cdots+\l_{k-1}}(p_1,\cdots p_i, p_{i+1}, \cdots, p_k)= sgn(\sigma) \epsilon(\sigma){l_k}_{\l_{\sigma(1)},\cdots,\l_{\sigma(k-1)}}(p_{\sigma(1)},\cdots,p_{\sigma(i+1)},p_{\sigma(i)},\cdots,p_{\sigma(k)})|_{\l_k \mapsto \l_k^\dagger}\end{align*}
	where $\l_{k}^{\dagger}= -\sum_{i=1}^{k-1}\l_{i}-\p$.%, for all $\sigma\in \mathbb{S}_k$.
 \item \textbf{Higher conformal Jacobi identity}, i.e., for any $k\in N$ and homogeneous element $p_1, p_2, \cdots p_k \in \L$
 \begin{align*}
 \sum_{i+j=k+1}\sum_{\sigma\in S_k}sgn(\sigma)\epsilon(\sigma)(-1)^{j(i-1)}{l_s}_{\l_{\sigma(1)}+\l_{\sigma(2)}+\cdots+\l_{\sigma(t)},\l_{\sigma(t+1)},\cdots,\l_{\sigma(k)}}({l_t}_{\l_{\sigma(1)},\cdots,\l_{\sigma(t-1)}}&(p_{\sigma(1)},\cdots, p_{\sigma(t)}),\\& p_{\sigma(t+1)}, \cdots, p_{\sigma(k)}).
 \end{align*}Where $\sigma\in S_{t,k-t}$ means that either $\sigma(t)=k$ or $\sigma(k)=k$. It have effect on $\l_{\sigma}^{\dagger}$, i.e., if $\sigma(t)=k$, then we use the notation $\l_{\sigma}^{\dagger}=\l_{\sigma(1)}+\cdots+\l_{\sigma(t)}$ and when $\sigma(k)=k$, then we use the notation $\l_{\sigma}^{\dagger}=-\p-\l_{\sigma(t+1)}-\cdots-\l_{\sigma(k)}.$ 
 \end{enumerate}
\end{defn} Note that Lie conformal algebra and differential graded Lie conformal algebra are particular cases of the $\L_\infty$-conformal algebra. However, we are also familiar with the $2$-term $\L_{\infty}$-algebra, defined in \cite{ADSS} that is a particular case of $\L_\infty$-conformal algebra. In the following definition, we define a $2$-term $\L_\infty$-conformal algebra: 
\begin{defn}
 A \textbf{$2$-term $\L_\infty$-conformal algebra} is a triple consisting of 
 \begin{enumerate}\item[(i)] a 
 chain complex of $\mathbb{C}[\p]$-modules $d:\L_1\to \L_0$,
 \item[(ii)] a conformal sesqui-linear and skew-symmetric $\mathbb{C}$-bilinear map $ \llbracket \cdot_\l \cdot \rrbracket : \L_i\otimes \L_j\to \L_{i+j} \llbracket \l \rrbracket $, for $i,j,i+j\in [ 0, 1]$, and
 \item[(iii)] a conformal sesqui-linear skew-symmetric trilinear map $l_3:\L_0\otimes \L_0\otimes \L_0\to \L_{1} \llbracket \l,\m \rrbracket $,
 \end{enumerate}
that satisfy the following set of identities :
\begin{enumerate}
 \item[(L1)] $ \llbracket m_\l n \rrbracket =0$,
 \item[(L2)] $ \llbracket p_\l m \rrbracket =- \llbracket m_{-\p-\l}p \rrbracket $,
 \item[(L3)] $ \llbracket p_\l q \rrbracket =- \llbracket q_{-\p-\l}p \rrbracket $,
\item [(L4)]$d ( \llbracket p_\l m \rrbracket )= \llbracket p_\l d(m) \rrbracket $ ,
\item [(L5)]$ \llbracket d(m)_\l n \rrbracket = \llbracket m_\l d(n) \rrbracket $ ,
 \item [(L6)] $d((l_3)_{\l,\m}(p,q,r))= ( \llbracket p_\l ( \llbracket q_\m r \rrbracket ) \rrbracket - \llbracket ( \llbracket p_\l q \rrbracket )_{\l+\m},r \rrbracket - \llbracket q_\m ( \llbracket p_\l r \rrbracket ) \rrbracket $,
 \item [(L7)]$(l_3)_{\l,\m}(p,q,d(m))= \llbracket p_\l ( \llbracket q_\m m \rrbracket ) \rrbracket - \llbracket ( \llbracket p_{\l}q \rrbracket )_{\l+\m}m \rrbracket - \llbracket q_\m ( \llbracket p_\l m \rrbracket ) \rrbracket $,
 \item[(L8)] \begin{align*}
 & \llbracket p_\l {l_3}_{\m,\nu}(q,r,w) \rrbracket - \llbracket q_\m {l_3}_{\l,\nu}(p,r, w) \rrbracket + \llbracket r_{\nu} {l_3}_{\l,\m}(p,q, w) \rrbracket - \llbracket w_{-\p-\l-\m-\nu} {l_3}(p,q, r) \rrbracket \\&= {l_3}_{\l+\m,\nu}( \llbracket p_\l q \rrbracket ,r,w) +{l_3}_{\m,\l+\nu}(q, \llbracket p_\l r \rrbracket ,w)+ {l_3}_{\m,\nu}(q,r, \llbracket p_\l w \rrbracket ) \\& +{l_3}_{\l,\m+\nu}(p, \llbracket q_\m r \rrbracket , w) - {l_3}_{\l,\nu}(p,r, \llbracket q_\m w \rrbracket )+ {l_3}_{\l,\m}(p,q , \llbracket r_\nu w \rrbracket ),
 \end{align*} for all $p,q,r,w\in \L_0$ and $m,n\in \L_1$.\end{enumerate}\end{defn} It follows from the above definition that $2$-term $\L_\infty$-conformal algebra is nothing but an $\L_\infty$-conformal algebra, whose underlying graded $\mathbb{C}[\p]-$module is concentrated in degree $0$ and $1$, i.e., $\L=\L_0+\L_1$. This implies that $l_k=0$ for $k\geq 4$.\\
 In the following, we define the \textbf{conformal morphisms of the two $2$-term $\L_\infty$-conformal algebras.}
 \begin{defn}
 Let $\L =(d:\L_1\to \L_0, \llbracket \cdot_\l \cdot \rrbracket , l_3)$ and $\L' =(d':\L_1'\to \L'_0, \llbracket \cdot_\l \cdot \rrbracket ', l'_3)$ are two $2$-term $\L_\infty$-conformal algebras. A homomorphism of these conformal algebras is given by a triple $(f_0,f_1,f_2)$, where $f_0: \L_0\to \L'_0$, $f_1: \L_1\to \L'_1$ are $\mathbb{C}[\p]$-linear maps, and $f_2: \L_0\otimes \L_0\to \L_1'[\l]$ is a conformal sesqui-linear and skew-symmetric map such that following equations hold for all $p, q, r \in \L_0$ and $m \in \L_1$:
 \begin{enumerate}
 \item[(H1)] $f_{0} \circ d = d'\circ f_1,$ 
 \item[(H2)] $d'({f_2}_\l(p,q)) = \llbracket {f_0(p)}_\l f_0(q) \rrbracket - f_0( \llbracket p_\l q \rrbracket )$,
 \item[(H3)] ${f_2}_\l(p, d(m)) =- f_1( \llbracket p_\l m \rrbracket ) + \llbracket f_0(p)_\l f_1(m) \rrbracket ,$ 
 \item[(H4)] ${f_2}_\l(d(m),p) =- f_1( \llbracket m_\l p \rrbracket ) + \llbracket {f_1(m)}_\l f_0(p) \rrbracket ,$
 \item[(H5)] \begin{align*}{l_3}_{\l,\m}(f_0(p), f_0(q), f_0(r)) - f_1 {l_3}_{\l,\m}(p, q, r)&= \llbracket {f_0(p)}_\l {f_2}_\m(q, r) \rrbracket + {f_2}_\l (p, ( \llbracket {q}_\m r \rrbracket ))\\& - \llbracket {f_0(q)}_\m {{f_2}_\l(p, r)} \rrbracket -{f_2}_\m (q, \llbracket {f_0(p)}_\l r \rrbracket )\\& - \llbracket {{f_2}_\l (p,q)}_{\l+\m } f_0 (r) \rrbracket - {f_2}_{\l+\m}( \llbracket {f_0(p)}_\l q \rrbracket , r). \end{align*}
 \end{enumerate}
 \end{defn}
 Further suppose that $\L =(d:\L_1\to \L_0, \llbracket \cdot_\l \cdot \rrbracket , l_3)$ be $2$-term $\L_\infty$-conformal algebra, then identity homomorphism $Id_\L:\L\to \L$ is given by a triple $Id_\L= (Id_{\L_0}, Id_{\L_1},0)$, It corresponds the following identities:
 \begin{itemize}
 \item $f_{0}\circ d = d\circ f_1,$
 \item $0 = \llbracket {f_0(p)}_\l f_0(q) \rrbracket - f_0( \llbracket p_\l q \rrbracket ),$
 \item $0 =- f_1( \llbracket p_\l m \rrbracket ) + \llbracket f_0(p)_\l f_1(m) \rrbracket ,$ 
 \item $0 =- f_1( \llbracket m_\l p \rrbracket ) + \llbracket {f_1(m)}_\l f_0(p) \rrbracket , $
 \item ${l_3}_{\l,\m}(f_0(p), f_0(q), f_0(r)) - f_1 {l_3}_{\l,\m}(p, q, r) = 0 .$
 \end{itemize}
%If we have two homomorphisms of $2$-term $\L_\infty$-conformal algebras from $\mathfrak{F}:\L\to \L'$ and $\mathfrak{G}:\L'\to \L''$ given by $(f_0,f_1,f_2)$ and $(g_0,g_1,g_2)$ respectively. Then their composition is given by $\mathfrak{F}\mathfrak{G}=(g_0\circ f_0, g_1\circ f_1,f_2\circ(f_0\otimes f_0)+g_1 \circ f_2)$. One can verify this result??????

\begin{defn}
 Let $(\L_1 \xrightarrow{d} \L_0, \llbracket \cdot_\l\cdot \rrbracket , l_3)$ be a $2$-term $\L_\infty$-algebra. A \emph{homotopy Nijenhuis operator} on this $2$-term $\L_\infty$-conformal algebra is a triple $\N = (\N_0, \N_1, \N_2)$ consisting of two $\mathbb{C}[\p]$-linear maps $\N_0 : \L_0 \to \L_0$, $\N_1 : \L_1 \to \L_1$, and a conformal sesquilinear skew-symmetric map $\N_2 : \L_0 \times \L_0 \to \L_1[\l]$, subject to satisfying the following set of identities:
\begin{align}
 &d \circ \N_1 = \N_0 \circ d, \label{eq:identity1} \\
 &d({\N_2}_\l(p, q)) = \N_0 \Big( \llbracket \N_0(p)_\l q \rrbracket + \llbracket p_\l \N_0(q) \rrbracket - \N_0( \llbracket p_\l q \rrbracket ) \Big) - \llbracket \N_0(p)_\l \N_0(q) \rrbracket , \label{eq:identity2} \\&
 {\N_2}_\l( p, d (m) ) = \N_1 \Big( \llbracket \N_0(p)_\l m \rrbracket + \llbracket p_\l \N_1(m) \rrbracket - \N_1( \llbracket p_\l  m \rrbracket ) \Big) - \llbracket \N_0(p)_\l  \N_1(m) \rrbracket , \label{eq:identity3} \\&
 \llbracket \N_0(p) _\l  {\N_2}_\m(q, r) \rrbracket - \llbracket \N_0(q) _\m  {\N_2}_{\l}(p, r) \rrbracket - \llbracket \N_2(p, q)_{\l+\m }\N_0(r) \rrbracket 
  - {\N_2}_{\l+\m} \Big( \llbracket \N_0(p)_\l q \rrbracket + \llbracket p_\l \N_0(q) \rrbracket - \N_0 ( \llbracket p_\l q \rrbracket ), r \Big) \nonumber \\&
  + {\N_2}_{\l} \Big(  p, \llbracket \N_0(q)_\m r \rrbracket + \llbracket q_\m \N_0(r) \rrbracket - \N_0 ( \llbracket q_\m r \rrbracket ) \Big) \nonumber 
  - {\N_2}_\m \Big( q, \llbracket \N_0(p)_\l r \rrbracket + \llbracket p_\l  \N_0(r) \rrbracket - \N_0 ( \llbracket p_\l r \rrbracket ) \Big) \nonumber \\&
  - \N_1 \Big( \llbracket p_\l {\N_2}_{\m}(q, r) \rrbracket - \llbracket q_\m {\N_2}_\l (p, r) \rrbracket - \llbracket {\N_2}_\l(p, q)_{\l+\m} r  \rrbracket \nonumber  
 - {\N_2}_{\l+\m}( \llbracket p_\l q \rrbracket , r) + {\N_2}_\l ( p, \llbracket q_\m  r \rrbracket  ) - {\N_2}_\m( q, \llbracket p_\l r \rrbracket) \Big) \nonumber  \\ &= {l_3}_{\l,\m}(\N_0(p), \N_0(q), \N_0(r))- \N_1 \Big( {l_3}_{\l,\m}(\N_0(p), \N_0(q), r) 
 +{l_3}_{\l,\m}(\N_0(p), q, \N_0(r)) + {l_3}_{\l,\m}(p, \N_0(q), \N_0(r)) \Big)
 \nonumber \\&
 + \N_1^2 \Big( {l_3}_{\l,\m}(\N_0(p), q, r) + {l_3}_{\l,\m}(p, \N_0(q), r) + {l_3}_{\l,\m}(p, q, \N_0(r)) \Big) 
\nonumber\\&- \N_1^3 {l_3}_{\l,\m}(p, q, r), \label{eq:identity4}
\end{align}
for all $p, q, r \in \L_0$ and $m \in \L_1$.
\end{defn}
 A $2$-term Nijenhuis $\L_\infty$-conformal algebra is a $2$-term $\L_\infty$-conformal algebra $\L =(d:\L_1\to \L_0, \llbracket \cdot_\l \cdot \rrbracket , l_3)$ equipped with a homotopy Nijenhuis operator $\N =(\N_0, \N_1, \N_2)$ on it. We denote a $2$-term Nijenhuis $\L_\infty$-conformal algebra as above by $\L_\N = (d: \L_1\to \L_0, \llbracket \cdot_\l \cdot \rrbracket , l_3, \N_0, \N_1, \N_2)$ or simply by $\L_\N$. 
\subsection{ Skeletal $2$-terms Nijenhuis $\L_\infty$-conformal algebras}
\begin{defn} Let $\L_\N$ be a $2$-term Nijenhuis $\L_\infty$-conformal algebra. It is said to be \textbf{skeletal} if $d =0$ and is said to be \textbf{strict} if $l_3 =0$ and $\N_2 =0$.
\end{defn}
Note that, if $\L_\N$ is skeletal, then we have 
\begin{enumerate}
 \item[(sk1)] $ \llbracket m_\l n \rrbracket =0$,
 \item[(sk2)] $ \llbracket p_\l m \rrbracket =- \llbracket m_{-\p-\l}p \rrbracket $,
 \item[(sk3)] $ \llbracket p_\l q \rrbracket =- \llbracket q_{-\p-\l}p \rrbracket $,
 \item [(sk4)] $0= ( \llbracket p_\l ( \llbracket q_\m r \rrbracket ) \rrbracket - \llbracket ( \llbracket p_\l q \rrbracket )_{\l+\m},r \rrbracket - \llbracket q_\m ( \llbracket p_\l r \rrbracket ) \rrbracket $,
 \item [(sk5)]$0= \llbracket p_\l ( \llbracket q_\m m \rrbracket ) \rrbracket - \llbracket ( \llbracket p_{\l}q \rrbracket )_{\l+\m}m \rrbracket - \llbracket q_\m ( \llbracket p_\l m \rrbracket ) \rrbracket $,
 \item[(sk6)] \begin{align*}
 & \llbracket p_\l {l_3}_{\m,\nu}(q,r,w) \rrbracket - \llbracket q_\m {l_3}_{\l,\nu}(p,r, w) \rrbracket + \llbracket r_{\nu} {l_3}_{\l,\m}(p,q, w) \rrbracket - \llbracket w_{-\p-\l-\m-\nu} {l_3}(p,q, r) \rrbracket \\&= {l_3}_{\l+\m,\nu}( \llbracket p_\l q \rrbracket ,r,w) +{l_3}_{\m,\l+\nu}(q, \llbracket p_\l r \rrbracket ,w)+ {l_3}_{\m,\nu}(q,r, \llbracket p_\l w \rrbracket ) \\& +{l_3}_{\l,\m+\nu}(p, \llbracket q_\m r \rrbracket , w) - {l_3}_{\l,\nu}(p,r, \llbracket q_\m w \rrbracket )+ {l_3}_{\l,\m}(p,q , \llbracket r_\nu w \rrbracket ),
 \end{align*} 
 
 \item [(sk7)]\begin{align*}
 0&= \N_1 \Big( \llbracket \N_0(p)_\l m \rrbracket + \llbracket p_\l \N_1(m) \rrbracket - \N_1( \llbracket p_\l  m \rrbracket ) \Big) - \llbracket \N_0(p)_\l  \N_1(m) \rrbracket ,
\end{align*}
\item[(sk8)] \begin{align*}
 0 &= \N_0 \Big( \llbracket \N_0(p)_\l q \rrbracket + \llbracket p_\l \N_0(q) \rrbracket - \N_0( \llbracket p_\l q \rrbracket ) \Big) - \llbracket \N_0(p)_\l \N_0(q) \rrbracket , 
\end{align*}
\item[(sk9)]\begin{align*} &\llbracket \N_0(p) _\l  {\N_2}_\m(q, r) \rrbracket - \llbracket \N_0(q) _\m  {\N_2}_{\l}(p, r) \rrbracket - \llbracket \N_2(p, q)_{\l+\m }\N_0(r) \rrbracket 
  - {\N_2}_{\l+\m} \Big( \llbracket \N_0(p)_\l q \rrbracket + \llbracket p_\l \N_0(q) \rrbracket - \N_0 ( \llbracket p_\l q \rrbracket ), r \Big) \nonumber \\&
  + {\N_2}_{\l} \Big(  p, \llbracket \N_0(q)_\m r \rrbracket + \llbracket q_\m \N_0(r) \rrbracket - \N_0 ( \llbracket q_\m r \rrbracket ) \Big) \nonumber 
  - {\N_2}_\m \Big( q, \llbracket \N_0(p)_\l r \rrbracket + \llbracket p_\l  \N_0(r) \rrbracket - \N_0 ( \llbracket p_\l r \rrbracket ) \Big) \nonumber \\&
  - \N_1 \Big( \llbracket p_\l {\N_2}_{\m}(q, r) \rrbracket - \llbracket q_\m {\N_2}_\l (p, r) \rrbracket - \llbracket {\N_2}_\l(p, q)_{\l+\m} r  \rrbracket \nonumber  
 - {\N_2}_{\l+\m}( \llbracket p_\l q \rrbracket , r) + {\N_2}_\l ( p, \llbracket q_\m  r \rrbracket  ) - {\N_2}_\m( q, \llbracket p_\l r \rrbracket) \Big) \nonumber  \\ &= {l_3}_{\l,\m}(\N_0(p), \N_0(q), \N_0(r))- \N_1 \Big( {l_3}_{\l,\m}(\N_0(p), \N_0(q), r) 
 +{l_3}_{\l,\m}(\N_0(p), q, \N_0(r)) + {l_3}_{\l,\m}(p, \N_0(q), \N_0(r)) \Big)
 \nonumber \\&
 + \N_1^2 \Big( {l_3}_{\l,\m}(\N_0(p), q, r) + {l_3}_{\l,\m}(p, \N_0(q), r) + {l_3}_{\l,\m}(p, q, \N_0(r)) \Big) 
\nonumber\\&- \N_1^3 {l_3}_{\l,\m}(p, q, r),
\end{align*}
 for all $p,q,r,w\in \L_0$ and $m,n\in \L_1$.
 \end{enumerate}
The following result gives a characterization of skeletal $2$-term Nijenhuis $\L_\infty$-conformal algebras in terms of $3$-cocycles of Nijenhuis Lie conformal algebras. 
From the definition of skeletal Nijenhuis Lie conformal algebra, it is clear that identities $sk3, sk4, sk8$ on $\mathbb{C}[\p]$-module $\L_0$ show that the $(\L_0, \llbracket \cdot_\l \cdot \rrbracket , \
\N_0)$ is a Lie conformal algebra denoted by ${\L_0}_{\N_0}$. On the other hand, $sk2, sk5, sk7$ on $\mathbb{C}[\p]$ module $\L_1$ show that ${\L_1}_{\N_1}=(\L_1, \llbracket \cdot_\l \cdot \rrbracket , \N_1)$ is conformal representation with the representation map $\rho:\L_0\times \L_1\to \L_1[\l]$, given by $\rho(p)_\l m= \llbracket {p}_\l m \rrbracket $ for $p\in \L_0$ and $m\in \L_1$.
Furthermore, conditions $sk6$ and $sk9$ are equivalent to writing 
\begin{align*}
 &\delta(l_3) (p,q,r,w)=0,\\
 &d_{\N_0,\N_1}^2(\N_2)(p,q,r)=\xi^3_{\N_0,\N_1}{l_3}_{\l,\m}(p,q,r). \end{align*} 
 Therefore, $$d^3_{\N L}(l_3, \N_2) = ( \delta^3(l_3), -\xi^3_{\N_0,\N_1}(l_3) + d^2_{\N_0,\N_1}(\N_2)).$$
Hence $(l_3, \N_2)\in C^3_{\N L}(\L_0,\L_1)$ is a $3$-cocycle of the Nijenhuis Lie
conformal algebra $(\L_0, \llbracket \cdot_\l\cdot \rrbracket ,\N_0)$ with coefficients in the Nijenhuis representation $(\L_1, \rho,\N_1)$. The triple $({\L_0}_{\N_0}, {\L_1}_{\N_1}, (l_3, \N_2))$, we obtain has $1-1$ correspondence with the $2$-term Nijenhuis $\L_\infty$-conformal algebra. The converse to this fact can easily be verified.
%{\color{red}Conversely, let $(\L_P, \mathfrak{M}_\phi, (g, \tau))$ be a triple in which $\L_P$ is a Nijenhuis Lie conformal algebra, $\mathfrak{M}_\phi$ is a conformal representation (with the action map $\rho : \L \to Cend(M))$ and $(f, \theta) \in C^3_{AL}(\L_P, \mathfrak{M}_\phi)$ is a $3$-cocycle. Then it is easy to verify that $(0:M\to \L, \llbracket \cdot_\l \cdot \rrbracket , g, P, \phi, \tau)$ is a skeletal $2$-term Nijenhuis $\L_\infty$-conforml algebra, where the bilinear map , is given by $ \llbracket p_\l y \rrbracket := [p_\l y]_\L$, $ \llbracket p_\l v \rrbracket =- \llbracket v_{-\p-\l}p \rrbracket :=\rho(p)_\l v$ and $u,v :=0$, for $p,y \in \L, u,v\in M.$ This completes the proof.}
Thus, we have the following Proposition:
\begin{prop}There is a one-to-one correspondence between skeletal $2$-term Nijenhuis $\L_\infty$- conformal algebras and third cocycles of Nijenhuis Lie conformal algebras with coefficients in Nijenhuis representations. \end{prop}
 The above result motivates us to consider the following notion. Let $\L_\N$ and $\L'_{\N'}$ are two skeletal $2$-term Nijenhuis $\L_\infty$-conformal algebras on the same chain complex. They are said to be equivalent if $ \llbracket \cdot_\l \cdot \rrbracket = \llbracket \cdot_\l \cdot \rrbracket ' $, $\N_0=\N_0'$, $\N_1=\N_1'$ and there exists a conformal skew-symmetric bilinear map $f : \L_0 \otimes \L_0 \to \L_1[\l]$ and a $\mathbb{C}$-linear map $\xi :\L_0 \to \L_1$ such that $(l'_3,\N'_2)=(l_3,\N_2)+d^2_{\N L}((f, \xi))$, where $d^2_{\N L}$ is the coboundary operator of the Nijenhuis Lie conformal algebra ${\L_0}_{\N_0}$ with coefficients in the representation ${\L_1}_{\N_1}$. Thus, we arrive at the following theorem:

\begin{thm}There is a $1 -1$ correspondence between the equivalence class of skeletal $2$-term Nijenhuis $\L_\infty$-conformal algebras $\L_\N$ and triples of the form $(\L_\N, \M_\Q, (g,\tau))$, where $\L_\N $ is a Nijenhuis Lie conformal algebra, $\M_\Q$ is a representation and $(g, \tau)\in H^3_{\N L}(\L_\N,\M_\Q)$ is a $3rd$-cohomology class.\end{thm} 
 \subsection{ Strict $2$-terms Nijenhuis $\L_\infty$-conformal algebras}
 
 Note that $\L_\N$ is {strict} if $l_3 =0 $ and $\N_2=0$, then we have  
\begin{enumerate}
 \item[(st1)] $ \llbracket m_\l n \rrbracket =0$,
 \item[(st2)] $ \llbracket p_\l m \rrbracket =- \llbracket m_{-\p-\l}p \rrbracket $,
 \item[(st3)] $ \llbracket p_\l q \rrbracket =- \llbracket q_{-\p-\l}p \rrbracket $,
\item [(st4)]$d ( \llbracket p_\l m \rrbracket )= \llbracket p_\l d(m) \rrbracket $ ,
\item [(st5)]$ \llbracket d(m)_\l n \rrbracket = \llbracket m_\l d(n) \rrbracket $ ,
 \item [(st6)] $0= ( \llbracket p_\l ( \llbracket q_\m r \rrbracket ) \rrbracket - \llbracket ( \llbracket p_\l q \rrbracket )_{\l+\m},r \rrbracket - \llbracket q_\m ( \llbracket p_\l r \rrbracket ) \rrbracket $,
 \item [(st7)]$0= \llbracket p_\l ( \llbracket q_\m m \rrbracket ) \rrbracket - \llbracket ( \llbracket p_{\l}q \rrbracket )_{\l+\m}m \rrbracket - \llbracket q_\m ( \llbracket p_\l m \rrbracket ) \rrbracket $,
 \item [(st8)]$d \circ \N_1 = \N_0 \circ d,$
 \item [(st9)] $0 = \N_0 \Big( \llbracket {\N_0(p)}_\l q \rrbracket + \llbracket p_\l \N_0(q) \rrbracket - \N_0( \llbracket p_\l q \rrbracket ) \Big) - \llbracket \N_0(p)_\l \N_0(q) \rrbracket ,$
 \item [(st10)]$0 = \N_1 \Big( \llbracket {\N_0(p)}_\l m \rrbracket + \llbracket p_\l \N_1(m) \rrbracket - \N_1( \llbracket p_\l m \rrbracket ) \Big) - \llbracket {\N_0(p)}_\l \N_1(m) \rrbracket , $
 \end{enumerate}
 for all $p,q,r,w\in \L_0$ and $m,n\in \L_1$. 
 \par Next, we introduce the crossed module of Nijenhuis Lie conformal algebra and characterize strict $2$-term $\L_\infty$-conformal algebra. The concept of crossed module for a Nijenhuis Lie algebra was previously defined Das in \cite{ADas}. Here, we extend this notion by introducing the crossed module for Nijenhuis Lie conformal algebras.
 
\begin{defn} A crossed module of Nijenhuis Lie conformal algebras is a quadruple $( {\L_0}_{\N_0},{\L_1}_{\N_1}, t, \rho)$, where ${\L_1}_{\N_1}$ and $ {\L_0}_{\N_0}$ are both Nijenhuis Lie conformal algebras, $t: {\L_1}_{\N_1}\to {\L_0}_{\N_0}$ is a Nijenhuis Lie conformal algebra morphism and $\rho:{\L_0} \otimes {\L_1} \to {\L_1} [\l]$ is a conformal sesquilinear map that makes ${\L_1}_{\N_1}$ into a representation of the Nijenhuis Lie conformal algebra ${\L_0}_{\N_0}$ satisfying additional conditions:
 \begin{align*} t({\rho(p)}_\l m)=&{[p_\l t (m)]}_{\L _0},\\ 
 {\rho(t( m))}_\l n =&{[m_\l n]}_{\L_1},
 \end{align*} for all $p\in \L_0$ and $m, n \in \L_1$.
\end{defn}

Let $((\L_0, {[\cdot, \cdot]}_{\L_0}, \N_0), (\L_1, {[\cdot, \cdot]}_{\L_1}, \N_1), t, \rho)$ be a crossed module of Nijenhuis Lie conformal algebras. Then for any $ m, n \in \L_1 $,
We observe that
\begin{align*}
t\big({[m_\l n]}^{\N_1}_{\L_1}\big) &= t\big({[\N_1(m)_\l n]}_{\L_1} + {[m_\l \N_1(n)]}_{\L_1} - \N_1({[m_\l n]}_{\L_1})\big) \\
&= {[{t(\N_1(m))}_\l t(n)]}_{\L_0} + {[{t(m)}_\l t(\N_1(n))]}_{\L_0} - t\N_1({[m_\l n]}_{\L_1}) \\
&= {[{t(m)}_\l t(n)]}^{\L_0}_{\N_0} \quad (\because t\circ \N_1 = \N_0\circ t),
\end{align*}
which shows that $ t : \L_1^{\N_1} \to \L_0^{\N_0} $ is a homomorphism of deformed Lie conformal algebras. 
Next, we consider the map $ \rho_1 : \L_0^{\N_0} \otimes \L_1^{\N_1}\to \L_1^{\N_1} $ defined by 
${\rho_1 (p)}_\l (m) := \rho {(\N_0(p))}_\l m + \rho (p)_\l{\N_1(m)} - \N_1(\rho (p)_\l m), \quad \text{for } p \in \L_0^{\N_0}, m \in \L_1^{\N_1}.$
It is easy to see that $ \rho_1 $ defines a representation of the induced Lie conformal algebras. Moreover, for any $ p \in \L_0 $ and $ m, n \in \L_1 $, we have
\begin{align*}
t({\rho_1 (p)}_\l (m)) &= t\big({\rho (\N_0(p))}_\l m + {\rho(p)}_\l {\N_1(m)} - \N_1({\rho (p)}_\l m)\big) \\
&= {[{\N_0(p)}_\l t(m)]}_{\L_0} + {[p_\l t(\N_1(m))]}_{\L_0} - \N_0{[p_\l t(m)]}_{\L_0} \\
&= {[p_\l t(m)]}_{\L_0}^{\N_0},
\end{align*}
and
\begin{align*}
{\rho_1 (t(m))}_\l (n) &= {\rho (\N_0 t(m))}_\l n + {\rho( t(m))} _\l{\N_1(n)} - \N_1({\rho (t(m))}_\l n) \\
&= {[\N_1(m)_\l n]}_{\L_1} + {[m_\l \N_1(n)]}_{\L_1} - \N_1{([m_\l n])}_{\L_1} \\
&= {[m_\l n]}^{\N_1}_{\L_1}.
\end{align*}
This shows that the quadruple $ (\L_0^{\N_0}, \L_1^{\N_1}, t, \rho_1) $ is a crossed module of Lie conformal algebras in the sense~of~\cite{BC}. 
%More generally, for any $ l \geq 0 $, one can show that the quadruple $ (\L_0^{\N_0^l}, \L_1^{\N_1^l}, t, \rho l) $ is a crossed module of Lie algebras, where \[\rho l^p(m) = \rho^{\N_0^l(p)}m + \rho p^{\N_1^l(m)} - \N_1^l(\rho p m), \quad \text{for } p \in \L_0^{\N_0^l} \text{ and } m \in \L_1^{\N_1^l}.\]
 \begin{prop}\label{directsum}
 Let $({\L_1}_{\N_1}, {\L_0}_{\N_0}, t, \rho)$ be a crossed module of Nijenhuis Lie conformal algebras. Then $(\L_0 \oplus \L_1, \N_0 \oplus \N_1)$ is a Nijenhuis Lie conformal algebra, where $\L_0 \oplus \L_1$ is equipped with the bracket
\begin{align}\label{crossedmod}
 [{(p, m)}_\l (q, n)] := ({[p_\l q]}_{\L_0}, {\rho(p)}_\l n - {\rho(q)}_{-\p-\l} m + {[m_\l n]}_{\L_1}),
\end{align}
for $(p,m), (q,n) \in \L_0 \oplus \L_1$.
\end{prop}
\begin{proof}
 Since $\L_0, \L_1$ are both Lie conformal algebras and $\rho : \L_0 \otimes \L_1\to \L_1[\l]$ is a conformal sesquilinear map, it follows that $\L_0 \oplus \L_1$ is a Lie conformal algebra with the bracket \eqref{crossedmod}. Moreover, for any $(p, m), (q, n) \in \L_0 \oplus \L_1$, we have
\begin{align*}
[{(\N_0 \oplus \N_1)(p, m)}_\l &(\N_0 \oplus \N_1)(q, n)] \\= &[{(\N_0(p), \N_1(m))}_\l (\N_0(q), \N_1(n))] \\
=& ({[{\N_0(p)}_\l \N_0(q)]}_{\L_0}, {\rho (\N_0(p))}_\l \N_1(n) - {\rho(\N_0(q))}_{-\p-\l} \N_1(m) + {[{\N_1(m)}_\l \N_1(n)]}_{\L_1}) \\
=& (\N_0({[\N_0(p)_\l q]}_{\L_0}+ {[p_\l \N_0(q)]}_{\L_0} - \N_0{[p_\l q]}_{\L_0} ) , \N_1(\rho{(\N_0(p))}_\l n+\rho(p)_\l {\N_1(n)}-\N_1 (\rho(p)_\l n)) \\&- \N_1( \N_0(q)_{-\p-\l} m+\rho(q)_{-\p-\l} {\N_1(m)}-\N_1 (\rho(q)_{-\p-\l }m)) \\&+ \N_1({[\N_1(m)_\l n]}_{\L_1}+ {[m_\l \N_1(n)]}_{\L_1} - \N_1{[m_\l n]}_{\L_1}))\\
=& (\N_0 \oplus \N_1)([{(\N_0(p), \N_1(m))}_\l (q, n)]+[{(p,m)}_\l (\N_0(q), \N_1(n))]- (\N_0 \oplus \N_1) [{(p,m)}_\l (q, n)]) \\
=& (\N_0 \oplus \N_1)([{(\N_0 \oplus \N_1)(p,m)}_\l (q, n)]+[{(p,m)}_\l (\N_0 \oplus \N_1)(q,n)]- (\N_0 \oplus \N_1) [{(p,m)}_\l (q, n)]) .
\end{align*}
This shows that the map $\N_0 \oplus \N_1 : \L_0 \oplus \L_1 \to \L_0 \oplus \L_1$ is a Nijenhuis operator. This proves the result.
\end{proof}
\begin{thm}\label{thmcross}
There is a $1-1$ correspondence between strict $2$-term Nijenhuis $\L_\infty$-conformal algebras and crossed modules of Nijenhuis Lie conformal algebras.
\end{thm}
\begin{proof}
Let $\L_\N = (\L_1 \xrightarrow{t} \L_0, \llbracket \cdot_\l \cdot \rrbracket , l_3 = 0, \N_0, \N_1, \N_2 = 0)= (\L_1 \xrightarrow{t} \L_0, \llbracket \cdot _\l \cdot \rrbracket , \N_0, \N_1)$ be a strict $2$-term Nijenhuis $\L_\infty$-conformal algebra. Then it follows from $st3$, $st6$ and $st9$ that ${\L_0}_{\N_0}=(\L_0, \llbracket \cdot _\l \cdot \rrbracket ,\N_0)$ is a Nijenhuis-Lie conformal algebra. Next, we define a skew-symmetric conformal bracket ${[\cdot_\l \cdot]}_{\L_1} : \L_1 \times \L_1 \to \L_1[\l]$ by
\begin{align*}
{ [m_\l n]}_{\L_1} := \llbracket t(m)_\l n \rrbracket , \quad \text{for } m,n \in \L_1.
\end{align*}
From conditions $st2$, $st5$ and $st7$, we see that $(\L_1, {[\cdot_\l \cdot]}_{\L_1})$ is a Lie conformal algebra. Moreover, condition $st10$ yields $\N_1 : \L_1 \to \L_1$, a Nijenhuis operator. Hence ${\L_1}_{\N_1}$ is also a Nijenhuis Lie conformal algebra. On the other hand, the conditions $st4$ and $st8$ imply that $t : {\L_1}_{\N_1} \to {\L_0}_{\N_0}$ is a Nijenhuis Lie conformal algebra morphism. Finally, we define a $\mathbb{C}$-linear map $\rho : \L_0 \times \L_1 \to \L_1[\l]$ given by
\begin{align*}
 {\rho(p)}_\l m := \llbracket p_\l m \rrbracket , \quad \text{for } p \in \L_0, m \in \L_1.
\end{align*}
It follows from $st7$ and $st10$ that $\rho$ makes ${\L_1}_{\N_1}$ into a representation of the Nijenhuis Lie conformal algebra ${\L_0}_{\N_0}$. We also have
\begin{align*}
 t({\rho(p)}_\l m) = t( \llbracket p_\l m \rrbracket )= \llbracket p_\l t(m) \rrbracket \quad \text{and} \quad {\rho(t(m))}_\l n = \llbracket t(m)_\l n \rrbracket = {[m_\l n]}_{\L_1},
\end{align*}for $p \in \L_0, m, n \in \L_1$. Hence $( {\L_0}_ {\N_0},{\L_1}_ {\N_1}, t, \rho)$ is a crossed module of Nijenhuis Lie conformal algebras.

Conversely, let $({\L_1}_ {\N_1}, {\L_0}_ {\N_0}, t, \rho)$ be a crossed module of Nijenhuis Lie conformal algebras. Then it is easy to verify that $(\L_1 \xrightarrow{t} \L_0, \llbracket \cdot_\l \cdot \rrbracket , l_3 = 0, \N_0, \N_1, \N_2 = 0)$ is a strict $2$-term Nijenhuis $\L_\infty$-conformal algebra, where the bracket \( \llbracket \cdot_\l \cdot \rrbracket : \L_i \times \L_j \to \L_{i+j}[\l]\) (for \(0 \leq i, j \leq 1\)) is given by
\begin{align*}
 \llbracket p_\l q \rrbracket := {[p_\l q]}_{\L_0}, \quad \llbracket p_\l m \rrbracket = - \llbracket m_{-\p-\l} p \rrbracket := \rho(p)_\l m \quad \text{and} \quad \llbracket m_\l n \rrbracket := 0,
\end{align*}
for $p, q \in \L_0, m, n \in \L_1$. The above two correspondences are inverse to each other. This completes the proof.\end{proof}
Combining the Proposition \ref{directsum} and Theorem \ref{thmcross}, we get the following results.
\begin{prop}
 Let $\L_\N = (\L_1 \xrightarrow{t} \L_0, \llbracket \cdot_\l \cdot \rrbracket , l_3 = 0, \N_0, \N _1, \N _2 = 0)$ be a strict $2$-term Nijenhuis $\L_\infty$-conformal algebra. Then ${\L_0 \oplus \L_1}_ {\N_0 \oplus \N_1}$ is a Nijenhuis Lie conformal algebra with the Lie conformal bracket given by
\begin{align*}
 [{(p, m)}_\l (q, n)] := ( \llbracket p_\l q \rrbracket , \llbracket p_\l n \rrbracket - \llbracket q_{-\p-\l} m \rrbracket + \llbracket t(m)_\l n \rrbracket ),
\end{align*}
for \((p, m), (q, n) \in \L_0 \oplus \L_1\).
\end{prop}
\begin{ex}
 Let $\L_\N$ be a Nijenhuis Lie conformal algebra. Then $(\L_\N, \L_\N, \mathrm{Id}, \mathrm{ad})$ is a crossed module of Nijenhuis Lie conformal algebras, where $\mathrm{ad}$ denotes the adjoint representation. Therefore, it follows that,
$\left( \L \xrightarrow{\mathrm{Id}} \L, {[\cdot_\l \cdot ]}_\L, l_3 = 0, \N_0 = \N, \N_1 = \N, \N_2 = 0 \right)$ is a strict $2$-term Nijenhuis $\L_\infty$-conformal algebra.
\end{ex}
%{\color{red}More generally, let $\L_P$ be a Nijenhuis Lie algebra and $\H \subset \L$ be a Lie ideal that satisfies \(P(\H) \subset \H\). Then \((\H_P, \L_P, i, \mathrm{ad})\) is a crossed module of Nijenhuis Lie algebras, where \(i : \H \to \L\) is the inclusion map. Hence $(\H \xrightarrow{i} \L, {[ \cdot, \cdot ]}_\L, l_3 = 0, P_0 = P, P_1= P, P_2= 0)$ is a strict $2$-term Nijenhuis $\L_\infty$-algebra.As a particular case, we also get the following.}
\begin{ex}
 Let \(\L_\N, \H_\Q\) be two Nijenhuis Lie conformal algebras and $f : \L_\N \to \H_\Q$ be a Nijenhuis Lie conformal algebra morphism. Then $(\ker f, \L, inc, \mathrm{ad})$ is a crossed module of Nijenhuis Lie conformal algebras, where $inc: \ker f \to \L$ is the inclusion map.
\end{ex}
\section{ Non-Abelian Extension of Nijenhuis Lie conformal algebra}In this section, we study the non-abelian extensions of Nijenhuis Lie conformal algebra $\L_\N=(\L, {[\cdot_\l\cdot]}_\L, \N)$ by another Nijenhuis Lie conformal algebra $\H_\Q=(\H, {[\cdot_\l\cdot]}_\H, \Q)$. We show that the non-abelian extensions can be classified by the second non-abelian cohomology group of Nijenhuis Lie conformal algebras.
\begin{defn}
 Let $\L_\N$ and $\H_\Q$ be two Nijenhuis Lie conformal algebras. A non-abelian extension of $\L_\N$ by $\H_\Q$, is a Nijenhuis Lie conformal algebra $\E_\R=(\E,[\cdot_\l\cdot],\R)$ together with a short exact sequence
\begin{align}\label{nonab}
 0 \longrightarrow \H_\Q \xrightarrow{inc} \E_\R \xrightarrow{proj} \L_\N \longrightarrow 0.
\end{align} of Nijenhuis Lie conformal algebras.
Often we denote a non-abelian extension as above simply by $\E_\R$.\end{defn}\begin{defn}
Let $\E_\R$ and $\E'_{\R'}$ be two
extensions of $\L_\N$ by $\H_\Q$. They are called equivalent if there exists a momorphism $\tau : \E_\R\to\E'_\R$ of Nijenhuis Lie
conformal algebras 
such that the following diagram commutes
$$\begin{tikzcd}
0 \arrow{r} & \H_\Q \arrow{r}{inc} \arrow{d}[swap]{id} & \E_\R \arrow{r}{proj} \arrow{d}[swap]{\tau} & \L_\N \arrow{r}\arrow{d}[swap]{id} & 0\\
0 \arrow{r} & \H_\Q \arrow{r}{inc'} & \E'_{\R'} \arrow{r}{proj'} & \L_\N \arrow{r} &0.
\end{tikzcd}$$ The set of all equivalence classes of non-abelian extensions of $\L_\N$ by $\H_\Q$ is denoted by $Ext_{nab}(\L_\N , \H_\Q)$.
\end{defn}
 The non-abelian extension of Nijenhuis Lie conformal algebra defined here is also a $\mathbb{C}[\p]$-split extension. In this regard, we can require Nijenhuis Lie conformal algebra $\L_\N$ to be projective as $\mathbb{C}[\p]$-module. It is a generalization of $\mathbb{C}[\p]$-split abelian extension of Nijenhuis Lie conformal algebras.\begin{ex}
 Let $({\L_1}_{\N_1}, {\L_0}_{\N_0}, t, \rho)$ be a crossed module of Nijenhuis Lie conformal algebras. Then the exact sequence
$$0 \longrightarrow {\L_1}_{\N_1} \xrightarrow{inc}{\L_0}_{\N_0}\oplus{\L_1}_{\N_1}\xrightarrow{proj} {\L_0}_{\N_0} \longrightarrow 0.$$
is a non-abelian extension of ${\L_0}_{\N_0}$ by ${\L_1}_{\N_1}$, where the Nijenhuis Lie conformal algebra structure on ${\L_0\oplus\L_1}_{\N_0\oplus \N_1}$
is given in Proposition \ref{directsum}. Thus, it follows from Theorem \ref{thmcross} that a strict $2$-term Nijenhuis $\L_\infty$-conformal algebra gives rise to a non-abelian extension of Nijenhuis Lie conformal algebras.
\end{ex} Given a non-abelian extension ${\E_\R}$ of ${\L_\N}$ by ${\H_\Q}$, we can define the section map $s:{\L_\N} \to {\E_\R}$ of $proj$ that satisfy $proj\circ s=id_{\L_\N}$. Now we define two $\mathbb{C}[\p]$-module maps $\chi_\l:\L_\N\times\L_\N\to\H_\Q[\l]$ and $\rho: \L_\N \times \H_\Q \to \H_\Q[\l]$ and a $\mathbb{C}$-linear map $\Phi : \L_\N\to \H_\Q$ by
\begin{align}\label{maps1}
\chi_\l(p, q) &:= {[s(p)_\l s(q)]}_\E -s ({[p_\l q]}_\L)\\\label{maps2}
{\rho(p)}_\l(h) &:= {[s(p)_\l h]}_\E\\ \label{maps3}
\Phi(p) &:= \R(s(p)) - s(\N(p)).
\end{align}
Since the exact sequence \eqref{nonab} defines a non-abelian extension of the Nijenhuis Lie conformal algebra $\L$ by another Nijenhuis Lie conformal algebra $\H$
(by forgetting the Nijenhuis operators), it follows from \cite{F} that non-abelian $2$-cocycle of $\L$ with values in $\H$ satisfy following equations \begin{align}\label{2cocycle1}(\rho(p)_\l {\rho(q)}_\m - {\rho(q)}_\m {\rho(p)}_\l) (h) - {\rho([p_\l q])}_{\l+\m}(h) &= {[{\chi_\l(p, q)}_{\l+\m} h]}_\H\\
 \rho(p)_\l \chi_\m(q, r) + \rho(q)_\m \chi_\l(r, p) + \rho(r)_{-\p-\l-\m} \chi_\l(p, q)&\nonumber\\ \label{2cocycle2}- \chi_{\l+\m}({[q_\m r]}_\L, p) - \chi_{\l+\m}({[r_{-\p-\l} p]}_\L, q) - \chi_{\l+\m}({[p_\l q]}_\L, r) &= 0,
\end{align}for all $p, q, r \in \L$ and $ h \in \H$. In terms of $\L_\N$ and $\H_\Q$, the above expression can be expressed in the following form.
\begin{lem} The maps $\chi_\l$, $\rho$, and $\Phi$ defined above satisfy the following compatible conditions: for all $p, q \in \L$ and $h \in \H$,
\begin{align}\label{E1} &{\rho(\N(p))}_\l Q(h) = \Q\Big({\rho(\N(p))}_\l h+ {\rho(p)}_\l \Q(h)- \Q({(\rho(p)}_\l h)\Big) + Q({[\Phi(p)_\l h]}_\H) - {[\Phi(p)_\l Q(h)]}_\H ,
\end{align}
\begin{align}\label{E2}
&\nonumber\chi_\l(\N(p), \N(q)) - \Q(\chi_\l(\N(p), q)+ \chi_\l(p, \N(q)) - \Q \chi_\l(p, q)) - \Phi({[\N(p)_\l q]}_{\L}+{[p_\l \N(q)]}_{\L} -\N{[p_\l q]}_{\L} )\\&+ 
\rho(\N(p))_\l\Phi(q) - \rho(\N(q))_{-\p-\l}\Phi(p) + \Q\Big(\rho(q)_{-\p-\l}\Phi(p)-\rho(p)_\l \Phi(q) +\Phi({[p_\l q]}_{\L})\Big)+ {[\Phi(p)_\l \Phi(q)]}_{\H} = 0.
\end{align}
\end{lem} 
\begin{proof}Consider that
\begin{align*}
 \rho(\N(p))_\l \Q(h) &-\Q\Big({\rho(\N(p))}_\l h+ {\rho(p)}_\l \Q(h)- \Q({(\rho(p)}_\l h)\Big) - \Q{[\Phi(p)_\l h]}_\H + {[\Phi(p)_\l \Q(h)]}_\H
 \\&= {[s\N(p)_\l \Q(h)]}_{\E} - \Q\Big({[s\N(p)_\l h]}_{\E}+ {[s(p)_\l \Q(h)]}_{\E}-\Q {[s(p)_\l h]}_{\E}      \Big) \\&- \Q{[\R s(p)_\l h]}_{\E} + \Q{[s \N(p)_\l h]}_{\E} + {[\R s(p)_\l \Q(h)]}_{\E} - {[s\N(p)_\l \Q(h)]}_{\E}\\&
 =- \Q{[\R s(p)_\l h]}_{\E} - \Q  {[s(p)_\l \Q(h)]}_{\E}+\Q^2 {[s(p)_\l h]}_{\E}  + {[\R s(p)_\l \Q(h)]}_{\E} 
 \\&=- \Q\Big({[\R s(p)_\l h]}_{\E} +  {[s(p)_\l \Q(h)]}_{\E}-\Q({[s(p)_\l h]}_{\E}) \Big) + {[\R s(p)_\l \Q(h)]}_{\E} \quad (\text{as } \Q = {\R|}_{\H})\\&
= 0.\end{align*}
This proves the identities in \eqref{E1}. To prove the identity \eqref{E2}, we observe that
\begin{align*}
 &\chi_\l(\N(p), \N(q)) - \Q\Big(\chi_\l(\N(p), q)+ \chi_\l(p, \N(q)) - \Q \chi_\l(p, q)\Big) - \Phi\Big({[\N(p)_\l q]}_{\L}+{[p_\l \N(q)]}_{\L} -\N({[p_\l q]}_{\L}) \big)\\&+ 
\rho(\N(p))_\l\Phi(q) - \rho(\N(q))_{-\p-\l}\Phi(p) + \Q\Big(\rho(q)_{-\p-\l}\Phi(p)-\rho(p)_\l \Phi(q) +\Phi({[p_\l q]}_{\L})\Big)+ {[\Phi(p)_\l \Phi(q)]}_{\H}
\\&= {[s\N(p)_\l s\N(q)]}_{\E} - s({[\N(p)_\l \N(q)]}_{\L}) \\&- \Q\Big({[s\N(p)_\l s(q)]}_{\E} - s{[\N(p)_\l q]}_{\L}  
+{[s(p)_\l s\N(q)]}_{\E} - s{[p_\l \N(q)]}_{\L} -
\Q( {[s(p)_\l s (q)]}_{\E} - s{[p_\l q]}_{\L}) \Big)
 \\&
 - \R s\Big({[\N(p)_\l q]}_{\L}+{[p_\l \N(q)]}_{\L}-\N {[p_\l q]}_{\L}\Big)  + s\N\Big({[\N(p)_\l q]}_{\L}+{[p_\l \N(q)]}_{\L}-\N {[p_\l q]}_{\L}\Big)\\& 
+  {[s\N(p)_\l \R s(q)]}_{\E} - {[s \N(p)_\l s \N(q)]}_{\E}  - {[s\N(q)_{-\p-\l} \R s(p)]}_{\E} + {[s\N(q)_{-\p-\l} s\N(p)]}_{\E}\\& 
 + \Q{[s(q)_{-\p-\l}\R s(p)]}_{\E}
- \Q{[s(q)_{-\p-\l} s\N(p)]}_{\E}
- \Q{[s(p)_{\l}\R s(q)]}_{\E}
+ \Q{[s(p)_{\l} s\N(q)]}_{\E}
+ \Q\R s{[p_{\l}q]}_{\E}
- \Q s\N{[p_\l q]}_{\E}
\\& + {[\R s(p)_{\l} \R s(q)]}_{\E} - {[\R s(p)_\l s\N(q)]}_{\E} - {[s\N(p)_\l \R s(q)]}_{\E}+ {[s \N(p)_\l s \N(q)]}_{\E}
 \\&=0 .
\end{align*}The above expression vanishes as both $\N$ and $\R$ are Nijenhuis operators on $(\L,{[\cdot_\l\cdot]}_\L)$ and $(\E,{[\cdot_\l\cdot]}_\E)$ respectively. While canceling the terms, we use the property $\Q = {\R|}_{\H}$. This completes the proof.
\end{proof}
\begin{defn}
 A triple $(\chi_\l,\rho,\Phi)$ is called non-abelian $2$-cocycle of $\L$ with coefficients from $\H$ satisfying the Eqs. \eqref{2cocycle1}, \eqref{2cocycle2}, \eqref{E1} and \eqref{E2}. 
\end{defn}
For a split sequence in the category of $\mathbb{C}[\p]$-modules, the section is not unique in general. However, non-abelian $2$-cocycle depends on the section. For
two different sections, we have the following result.\\
Let $s': \L \to \E$ be any other section of the exact sequence in Eq. \eqref{nonab}. Define the map $\tau : \L \to \H$ by
\begin{align*}
 \tau(p) := s(p) - s'(p) \quad \text{for all } p \in \L.
\end{align*}
Let $\chi'_\l, \rho', \Phi'$ be the maps induced by the section $s'$, given by Eqs. \eqref{maps1}, \eqref{maps2}, and \eqref{maps3}. For all $p, q \in \L $ and $h \in \H$, we have:
\begin{align}\label{equivalent1}
 \rho(p)_\l h - \rho'(p)_\l h = {[s(p)_\l h]}_{\E} - {[s'(p)_\l h]}_{\E} = {[\tau(p)_\l h]}_{\H},
\end{align}
\begin{align}\label{equivalent2}
 \chi_\l(p, q) - \chi'_\l(p, q) &= 
 {[s(p)_\l s(q)]}_{\E} - s({[p_\l q]}_{\L}) - {[s'(p)_\l s'(q)]}_{\E} + s'({[p_\l q]}_{\L})\\&\nonumber
={[(s'+\tau)(p)_\l (s'+\tau)(q)]}_{\E} - (s'+\tau)({[p_\l q]}_{\L}) - {[s'(p)_\l s'(q)]}_{\E} + s'({[p_\l q]}_{\L})\\&\nonumber
={[s'(p)_\l s'(q)]}_{\E} +
{[s'(p)_\l \tau(q)]}_{\E}+
{[\tau(p)_\l s'(q)]}_{\E}+
{[\tau(p)_\l \tau(q)]}_{\E}\\&\nonumber
- s'({[p_\l q]}_{\L}) 
- \tau ({[p_\l q]}_{\L})  
- {[s'(p)_\l s'(q)]}_{\E} 
+s'({[p_\l q]}_{\L})\\&\nonumber
 = {[\tau(p)_\l \tau(q)]}_{\E}- \tau ({[p_\l q]}_{\L})
+{[s'(p)_\l \tau(q)]}_{\E}+
{[\tau(p)_\l s'(q)]}_{\E} 
 \\&\nonumber
 = {[\tau(p)_\l \tau(q)]}_{\E}- \tau ({[p_\l q]}_{\L})
+{[{s'(p)}_\l \tau(q)]}_{\E}-
{[{s'(q)}_{-\p-\l}\tau(p) ]}_{\E} 
\\&\nonumber =  {[\tau(p)_\l\tau(q)]}_{\H}- \tau({[p_\l q]}_{\L}) +\rho'(p)_\l \tau(q) - \rho'(q)_{-\p-\l} \tau(p),
\end{align} and
\begin{align}\label{equivalent3}
 \Phi(p) - \Phi'(p) &= (\R s - s\N)(p) - (\R{s'} - s'\N)(p)\\&\nonumber= \R(s - s')(p) - (s - s')\N(p)\\&\nonumber= \R \tau(p) - \tau \N(p).
\end{align}
The above discussion leads us to the following definition.
\begin{defn}
 Two non-abelian $2$-cocycles $(\chi_\l,\rho,\Phi)$ and $(\chi'_\l,\rho',\Phi')$ of $\L_\N$ with coefficients from $\H_\Q$ are said to be equivalent if there exists a map $\tau: \L \to \H$, that satisfies Eqs. \eqref{equivalent1}, \eqref{equivalent2} and \eqref{equivalent3}. The set of all equivalence classes of non-abelian cocycles of $\L_\N$ by $\H_\Q$ is denoted by $H^2_{nab}(\L_\N,\H_\Q).$ 
\end{defn}
%Second Non-abelian cohomology of $\L_\N$ with coefficients from $\H_\Q$ is denoted by $H^2_{nab}(\L_\N,\H_\Q).$ 
We have the following Theorem.
\begin{thm}\label{thm5.7}
 Let $\E_\R$ and $\E'_\R$ be two equivalent non-abelian extensions of $\L_\N$ by $\H_\Q$, with different sections $s$ and $s'$ respectively. Then the choice of different sections preserves the equivalence relation between two $2$-cocycles $(\chi_\l,\rho,\Phi)$ and $(\chi'_\l,\rho',\Phi')$.
\end{thm}
\begin{proof}
Let $\E_\R$ and $\E'_{\R'}$ be two equivalent non-abelian extensions of $\L_\N$ by $\H_\Q$. If $s: \L \to \E$ is a section of the map $proj$, then it is easy to observe that the map $s':= \tau \circ s$ is a section of the map $proj'$. Let $(\chi'_\l, \rho', \Phi')$ be the non-abelian $2$-cocycle corresponding to the non-abelian extension $\E'_{\R'}$ and the section $s'$. Then by using   ${\tau|}_{\H} = \text{Id}_{\H}$, we have \begin{align*}\chi'_\l(p, q) =& {[{s'(p)}_\l s'(q)]}_{\E'} - s'({[p_\l q]}_{\L}) 
= {[(\tau \circ s)(p)_\l (\tau \circ s)(q)]}_{\E'} - (\tau \circ s)({[p_\l q]}_{\L}) \\
=& \tau\Big({[s(p)_\l s(q)]}_{\E} - s({[p_\l q]}_{\L})\Big) = \chi_\l(p, q),\end{align*}
\begin{align*}{\rho'(p)}_\l(h) &= {[{s'(p)}_\l h]}_{\E} = {[{\tau \circ s(p)}_\l h]}_{\E'}
= \tau\left({[{s(p)}_\l h]}_{\E}\right) = {\rho(p)}_\l(h),
\end{align*}and
\begin{align*}\Phi'(p) &= \R{s'(p)} - s'\N(p)= \R(\tau \circ s(p)) - (\tau \circ s)\N(p)
= \tau\Big(\R s(p) - s\N(p)\Big) = \Phi(p) .
\end{align*}
Thus, we say both cocycles are equal, i.e., $(\chi_\l, \rho, \Phi) = (\chi'_\l, \rho', \Phi')$. Hence, they give rise to the same element in $H^2_{nab}(\L_\N,\H_\Q)$. Therefore, there is a well-defined map $\Upsilon : \text{Ext}_{nab}(\L_\N, \H_\Q) \to H^2_{nab}(\L_\N, \H_\Q)$.

Conversely, let $(\chi_\l, \rho, \Phi)$ be a non-abelian $2$-cocycle of $\L_\N$ with values in $\H_\Q$. Define $\E := \L \oplus \H$ with the conformal skew-symmetric  bilinear bracket
\begin{align*} {[{(p, h)}_\l (q, k)]}_{\E} := ({[p_\l q]}_{\L}, {\rho(p)}_\l k - {\rho(q)}_{-\p-\l} h + \chi_\l(p, q) + {[h_\l k]}_{\H})
\end{align*} for $(p, h), (q, k) \in \E $. Note that by using the conditions (8) and (9) in \cite{F}, the bracket ${[\cdot_\l \cdot ]}_{\E}$ satisfies the conformal Jacobi identity. In other words, $(\E, {[\cdot_\l \cdot]}_{\E})$ is a Lie conformal algebra. 

Further, we define a $\mathbb{C}[\p]$-module homomorphism $ \R : \E \to \E$ by
\begin{align*}
 \R(p, h) := (\N(p), \Q(h) + \Phi(p)),\quad \text{ for all } (p, h) \in \E .
\end{align*} Then 
\begin{align*}
 & {[{\R(p, h)}_\l \R(q, k)]}_{\E} \\
 &= {[{\Big( \N(p), \Q(h) + \Phi(p) \Big)}_\l \Big( \N(q), \Q(k) + \Phi(q) \Big)]}_{\E} \\
 &= \Bigg( {[{\N(p)}_\l \N(q)]}_\L ,
 \overbrace{{\rho (\N(p))}_\l \Q(k) }^{T1} 
 + \overbrace{{\rho (\N(p))}_\l \Phi(q)}^{T2}
 - \overbrace{{\rho (\N(q))}_{-\p-\l} \Q(h)}^{T3}
 - \overbrace{{\rho (\N(q))}_{-\p-\l} \Phi(p)}^{T4} \\
 &\quad + \overbrace{\chi_\l(\N(p), \N(q)) }^{T5}
 + \overbrace{{[{\Q(h)}_\l \Q(k)]}_\H} ^{T6}
 + \overbrace{{[{\Q(h)}_\l \Phi(q)]}_\H }^{T7}
 + \overbrace{{[{\Phi(p)}_\l \Q(k)]}_\H }^{T8}
 + \overbrace{{[{\Phi(p)}_\l \Phi(q)]}_\H}^{T9} \Bigg)\\
 &= \Bigg( \N({[{\N(p)}_\l q]}_\L + {[p_\l \N(q)]}_\L - \N ({[p_\l q]}_\L)), \overbrace{\Q({\rho (\N(p))}_\l k + {\rho (p)}_\l \Q(k) - \Q({\rho (p)}_\l k)) + {\Q[{\Phi(p)}_\l k]}_\H} ^{T1+T8}\\
 &\quad 
 - \overbrace{\Q({\rho (\N(q))}_{-\p-\l }h + {\rho( q)} _{-\p-\l }\Q(h) - \Q({\rho (q)}_{-\p-\l } h)) - \Q{[\Phi(q)_{-\p-\l } h]}_\H }^
 {T3+T7}\\
 &\quad + \overbrace{\Q({[\Q(h)_\l k]}_\H + {[h_\l \Q(k)]}_\H 
 - \Q {[h_\l k]}_\H )}^{T6}\\
 &\quad +\overbrace{\Q\big(\chi_\l(\N(p), q) + \chi_\l(p, \N(q)) - \Q\chi_\l(p, q)\big)  
 + \Phi({[\N(p)_\l q]}_\L + {[p_\l \N(q)]}_\L - \N{[p_\l q]}_\L)}^{T2+T4+T5+T9} \\
 &\quad 
 \overbrace{+ \Q\big(\rho (p)_\l \Phi(q) - \rho( q)_{-\p-\l} \Phi(p) - \Phi{[p_\l q]}_\L\big) \Bigg) }^{Continued ...T2+T4+T5+T9}\\ &
 = \Bigg( \N({[\N(p)_\l q]}_\L), \Q\big(\rho (\N(p))_\l k - \rho (q)_{-\p-\l} \Q(h) - \rho (q)_{-\p-\l} \Phi(p) + \chi_\l (\N(p), q) + {[(\Q(h) + \Phi(p))_\l k]}_\H\big) + \Phi({[\N(p)_\l q]}_\L) \Bigg) \\
 &\quad + \Bigg(\N{[p_\l \N(q)]}_\L, \Q\big(\rho (p)_\l \Q(k) + \rho (p)_\l \Phi(q) - \rho \N(q)_{-\p-\l}h + \chi_{\l}(p, \N(q)) + {[h_{\l} \Q(k) + \Phi(q)]}_\H\big) + \Phi{[p_{\l}\N(q)]}_\L \Bigg) \\
 &\quad - \Bigg( \N^2{[p_\l q]}_\L, \Q^2(\rho (p)_\l k - \rho (q)_{-\p-\l} h + \chi_{ \l}(p, q) + {[h_{\l} k]}_\H) + \Q \Phi{[p_{\l}q]}_\L + \Phi \N{[p_{\l} q]}_\L \Bigg) \\
 &= \R\Big({[\N(p)_\l q]}_\L, \rho (\N(p))_\l k - \rho (q)_{-\p-\l} \Q(h) - \rho (q)_{-\p-\l} \Phi(p) + \chi_\l(\N(p), q) + {[{(\Q(h) + \Phi(p))}_\l k]}_\H\Big) \\
 &\quad + \R\Big({[p_\l \N(q)]}_\L, \rho (p)_\l \Q(k) + \rho (p)_\l \Phi(q) - \rho (\N(q))_{-\p-\l} h + \chi_\l (p, \N(q)) + {[h_\l (\Q(k) + \Phi(q))]}_\H\Big) \\
 &\quad - \R\Big(\N{[p_\l q]}_\L, \Q(\rho (p)_\l k - \rho (q)_{-\p-\l} h + \chi_\l(p, q) + {[h_\l k]}_\H) + \Phi({[p_\l q]}_\L)\Big) \\
 &= \R\Big({[(\N(p), \Q(h) + \Phi(p))_\l (q, k)]}_{\E}\Big) + \R\Big({[(p, h)_\l (\N(q), \Q(k) + \Phi(q))]}_{\E}\Big) - \R^2\Big({[(p, h)_\l (q, k)]}_{\E}\Big) \\
 &= \R\Big({[\R(p, h)_\l (q, k)]}_{\E} + {[(p, h)_\l \R(q, k)]}_{\E}- \R({[(p, h)_\l (q, k)]}_{\E})\Big).
\end{align*}
Thus, we see that $\R$ is a Nijenhuis operator on $\E=\L\oplus \H$ and corresponding Nijenhuis Lie conformal algebra $(\E,{[\cdot_\l\cdot]}_\E, \R)$ is denoted by $\E_\R$.
Moreover, it is easy to see that
\begin{align*}
 0 \longrightarrow \H_\Q \xrightarrow{inc} \E_\R \xrightarrow{proj} \L_\N \longrightarrow 0
\end{align*}
is a non-abelian extension of the Nijenhuis Lie algebra $\L_\N$ by $\H_\Q$, where $inc(h) = (0, h)$ and $proj(p, h) = p$ for all $(p, h) \in \E$ and $h \in \H$.

Next, let $(\chi_\l, \rho, \Phi)$ and $(\chi'_\l, \rho', \Phi') $ be two equivalent non-abelian $2$-cocycles, i.e., there exists a linear map $\tau : \L \to \H$ such that the identities \eqref{equivalent1}, \eqref{equivalent2}, and \eqref{equivalent3} hold. Let $\E'_{\R'}$ be the Nijenhuis Lie conformal algebra induced by the $2$-cocycle $(\chi', \rho', \Phi')$. Note that the Lie conformal algebra $\E'=\L \oplus \H $ is a $\mathbb{C}[\p]$-module equipped with the Lie conformal bracket given by
\begin{align*}
 {[{(p, h)}_\l (q, k)]}_{\E'} := \Big({[p_\l q]}_{\L}, {\rho'(p)}_\l k - {\rho'(q)}_{-\p-\l} h + \chi'_\l(p, q) + {[h_\l k]}_{\H}\Big),
\end{align*} for $(p, h), (q, k) \in \E' $. Moreover, the map $R' : \E' \to \E'$ is given by
$\R'(p, h) := (\N(p), \Q(h) + \Phi'(p))$ for all $(p, h) \in \E'$. 
We now define a map $\phi : \L \oplus \H \to \L \oplus \H$ by
$\phi(p, h) := (p, h + \tau(p))$
for all $(p, h) \in \L \oplus \H$. Then, by a straightforward calculation, we show that
\begin{align*}
   \phi({[{(p, h)}_\l (q, k)]}_{\E}) =&\phi\Big({[p_\l q]}_{\L}, \rho(p)_\l k - {\rho(q)}_{-\p-\l} h + \chi_\l(p, q) + {[h_\l k]}_{\H}\Big) \\=&
\Big({[p_\l q]}_{\L}, \tau({[p_\l q]}_{\L})+ {\rho(p)}_\l k - {\rho(q)}_{-\p-\l} h + \chi_\l(p, q) + {[h_\l k]}_{\H}\Big)
\\=& \Big({[p_\l q]}_{\L}, \tau({[p_\l q]}_{\L})+ \rho'(p)_\l k +{[{\tau(p)}_\l k]}_\H- \rho'(q)_{-\p-\l} h -{[\tau(q)_{-\p-\l} h]}_\H \\&+\chi'_\l(p, q)+ {\rho'(p)}_\l \tau(q) - {\rho'(q)}_{-\p-\l} \tau(p) - \tau{[p_\l q]}_{\L} + {[{\tau(p)}_\l\tau(q)]}_{\H} + {[h_\l k]}_{\H}\Big) \\=& \Big({[p_\l q]}_{\L}, {\rho'(p)}_\l k + {\rho'(p)}_\l \tau(q) - {\rho'(q)}_{-\p-\l} \tau(p)
- {\rho'(q)}_{-\p-\l} h
+{[{\tau(p)}_\l k]}_\H \\&-{[{\tau(q)}_{-\p-\l} h]}_{\H}+ \chi'_\l(p, q) + {[{\tau(p)}_\l\tau(q)]}_{\H} + {[h_\l k]}_{\H}\Big) 
\\=&{[(p, h + \tau(p))_\l (q, k + \tau(q))]}_{\E'} \\=& {[{\phi(p, h)}_\l \phi(q, k)]}_{\E'}.
\end{align*}
Further,
\begin{align*} (\R' \circ \phi)(p, h) &=\R'(p, h + \tau(p))
\\&= (\N(p), \Q(h) + \Q(\tau(p)) + \Phi'(p))
\\&= (\N(p), \Q(h) + \tau(\N(p)) + \Phi(p)) \quad (\text{by Eq. \eqref{equivalent3}})
\\&= \phi(\N(p), \Q(h) + \Phi(p)) = (\phi \circ \R)(p, h).
\end{align*}
Hence, the map $\phi: \E\to\E$ defines an equivalence between the two non-abelian extensions. Therefore, we obtain a well-defined map $\Lambda : H^2_{\text{nab}}(\L_\N, \H_\Q) \to \text{Ext}_{\text{nab}}(\L_\N, \H_\Q) $. Finally, it is straightforward to verify that the maps $\Upsilon$ and $\Lambda$ are inverses of each other. This completes the proof.\end{proof}
\begin{rem}
    Theorem \ref{thm5.7} also holds for the deformed Nijenhuis Lie conformal algebras $(\L,{[\cdot_\l\cdot]}_\N)$. 
\end{rem}
\subsection{ Special Case: Abelian Extension }
 Considering the Nijenhuis Lie conformal algebra $\L_\N= (\L,{[\cdot_\l\cdot]}_\L,\N)$ and its representation $\M_\Q=(\M,\rho,\Q)$, we further study the abelian extension of Nijenhuis Lie conformal algebras. The $2$-coycle of Nijenhuis Lie conformal algebra is given by the triple $(\chi_\l,\rho,\Phi)$, where the corresponding maps $\chi_\l:\L\times \L\to \M[\l]$, $\rho:\L\times\M\to \M[\l]$ and $\Phi:\L\to \M$, satisfy Eqs. \eqref{2cocycle2} and \eqref{E2}\begin{align*}
 {\rho(p)}_\l \chi_\m(q, r) + {\rho(q)}_\m \chi_\l(r, p) + {\rho(r)}_{-\p-\l-\m} \chi_\l(p, q)- \chi_{\l+\m}({[q_\m r]}_\L, p) - \chi_{\l+\m}({[r_{-\p-\l} p]}_\L, q) - \chi_{\l+\m}({[p_\l q]}_\L, r) &= 0,
\\  \nonumber\chi_\l(\N(p), \N(q)) - \Q(\chi_\l(\N(p), q)+ \chi_\l(p, \N(q)) - \Q \chi_\l(p, q)) - \Phi({[{\N(p)}_\l q]}_{\L}+{[p_\l \N(q)]}_{\L} -\N{[p_\l q]}_{\L} )\\+ 
{\rho(\N(p))}_\l\Phi(q) - {\rho(\N(q))}_{-\p-\l}\Phi(p) + {\Q(\rho(q)}_{-\p-\l}\Phi(p)-{\rho(p)}_\l \Phi(q) +\Q{[p_\l q]}_{\L}) &= 0.
\end{align*}for all $p,q,r \in \L$.
Further, assume that we are having two cocycles 
 $(\chi_\l,\rho,\Phi)$ and $(\chi_\l,\rho,\Phi)$ of the Nijenhuis Lie conformal algebras with respect to its representation. They are said to be cohomologous if there exists a map $\tau:\L\to \M$ such that the following identities hold
 \begin{align}\label{equivalent21}
 \chi_\l(p, q) - \chi'_\l(p, q) &= {\rho(p)}_\l \tau(q) - \rho(q)_{-\p-\l} \tau(p) - \tau({[p_\l q]}_{\L}),
\end{align} and
\begin{align}\label{equivalent31}
 \Phi(p) - \Phi'(p) &= \Q \tau(p) - \tau \N(p).
\end{align}
 The second cohomology group associated to the defined $2$-cocycle is denoted by $H^2(\L_\N,\M_\Q)$.
 
A $2$-cocycle of a Nijenhuis Lie conformal algebra with coefficients in a Nijenhuis representation can be interpreted as a \textbf{non-abelian $2$-cocycle} by treating the tuple $(\M, \rho, \Q)$ as a Nijenhuis Lie conformal algebra equipped with the trivial Lie bracket on $\M$. Specifically, a triple $(\chi_\l, \rho, \Phi)$ is a $2$-cocycle of the Nijenhuis Lie conformal algebra $\L_\N$ with coefficients in the Nijenhuis representation $\M_\Q$ if and only if the triple $(\chi_\l, \rho, \Phi)$ is a non-abelian $2$-cocycle of the Nijenhuis Lie conformal algebra $\L_\N$ with values in the Nijenhuis Lie conformal algebra $(\M, {[\cdot, \cdot]}_\M = 0, \Q)$.

Next, let $\L_\N$ be a Nijenhuis Lie conformal algebra, and let $\M_\Q$ be a triple consisting of a $\mathbb{C}[\p]$-module $\M$ and a $\mathbb{C}$-linear map $\Q$, with trivial Lie conformal bracket on $\M$, i.e. ${[\cdot_\l\cdot]}_\M=0$. Then, an \textbf{abelian extension} of the Nijenhuis Lie conformal algebra $(\L, {[\cdot_\l \cdot]}_\L, \N)$ by the pair $(\M, {[\cdot_\l \cdot]}_\M=0, \Q)$ is defined as a short exact sequence of Nijenhuis Lie conformal algebras.
 \begin{align}\label{absh}
 0 \longrightarrow (\M, {[\cdot_\l \cdot]}_\M=0, \Q) \xrightarrow{inc} (\E, {[\cdot_\l \cdot]}_\E, \R) \xrightarrow{proj} (\L, {[\cdot_\l \cdot]}_\L, \N) \longrightarrow 0.
\end{align}
 Given an abelian extension $(\E, {[\cdot_\l \cdot]}_\E, \R)$ of $(\L, {[\cdot_\l \cdot]}_\L, \N) $ by $(\M, {[\cdot_\l \cdot]}_\M=0, \Q)$, we can define the section map $s:\L \to \E$ of $proj$. We can also define two $\mathbb{C}[\p]$-module maps $\chi_\l:\L \times\L \to\M [\l]$ and $\rho: \L \times \M \to \M[\l]$ and a $\mathbb{C}$-linear map $\Phi : \L\to \M$ by Eqs. \eqref{maps1}, \eqref{maps2}, and \eqref{maps3}
\begin{align*} 
\chi_\l(p,q) &:= {[s(p)_\l s(q)]}_\E -s ({[p_\l q]}_\L),\\
\rho(p)_\l(m) &:= {[s(p)_\l m]}_\E,\\
\Phi(p) &:= \R(s(p)) - s(\N(p)).
\end{align*} for $p,q\in \L$ and $m\in \M$. Similar to non-abelian extension, $\rho$ does not depend on the choice of the section map $s$. Moreover, by Eqs. \eqref{2cocycle1} and \eqref{E1}, the map $\rho$ makes the triple $(\M,\rho,\Q)$ into Nijenhuis representaion by considering trivial $\l$-bracket on $\M$, i.e, ${[\cdot_\l\cdot]}_\M=0$. From the Eqs. \eqref{2cocycle2} and \eqref{E2}, we see that the triple $(\chi_\l, \rho,\Phi)$ is a $2$-cocycle of Nijenhuis Lie conformal algebra with the coefficients in the Nijenhuis representation. The set of all equivalence classes of abelian extensions of the Nijenhuis Lie conformal algebra with the coefficients in the Nijenhuis representations is denoted by $Ext_{ab}(\L_\N,\M_\Q).$ Similar to Theorem \ref{thm5.7} of the non-abelian extensions, we have the following result.
\begin{thm}
Let $\L_\N$ be a Nijenhuis Lie conformal algebra and $\M_\Q$ be its representation. Then $Ext_{ab}(\L_\N,\M_\Q)\equiv H^2(\L_\N,\M_\Q)$. 
\end{thm} 

\section{ Automorphisms of Nijenhuis Lie conformal algebras and their inducibility }
In this section, we study the inducibility of a pair of automorphisms in the context of a non-abelian extension of Nijenhuis Lie conformal algebras and present the fundamental sequence of Wells. For this, we first define the automorphisms of Nijenhuis Lie conformal algebras, then establish the necessary and sufficient condition for these automorphisms to be inducible.

Let $\L_\N= (\L,{[\cdot_\l\cdot]}_\L,\N)$ and $\H_\Q= (\H,{[\cdot_\l\cdot]}_\H,\Q)$ be two Nijenhuis Lie conformal algebras and \begin{align*}
 0 \longrightarrow \H_\Q \xrightarrow{inc} \E_\R \xrightarrow{proj} \L_\N \longrightarrow 0
\end{align*} be a non-abelian extension of $\L_\N$ by $\H_\Q$ with the section $s$ of $proj$-map. %{\color{red} The corresponding non-abelian $2$-cocycles of $\E_\R$ is denoted by $(\chi_\l,\rho,\Phi)$.}
 Let $Aut_\H(\E_\R )$ be the set of all Nijenhuis Lie conformal algebra automorphisms $\gamma\in Aut(\E_\R)$ that satisfies ${\gamma|}_\H\subset \H$, i.e. $\H$ is invarient $\mathbb{C}[\p]$-module. More explicitly,
$Aut_\H(\E_\R ) := \{\gamma \in Aut(\E_\R)| \gamma(\H) = \H\}.$

For any section $s : \L\to\E$ of the $proj$-map , for any $\gamma \in Aut_\H(\E_\R)$, we can define a $\mathbb{C}[\p]$-module map $\bar{\gamma}: \L \to \L$ by $\bar{\gamma}(p) := proj\gamma s(p)$, for $p \in \L$. It is easy to verify that the map $\bar{\gamma}$ is independent of the choice of the section $s$. Moreover, $proj$ is a projection on $\L$ and $\gamma $ preserves $\L $, so $\bar{\gamma}$ is a bijection on $\L$. For any $p, q \in \L $, we have
\begin{align*}
 \bar{\gamma}({[p_\l q]}_\L) = proj {\gamma}(s{[p_\l q]}_\L) &= proj{\gamma}({[s(p)_\l s(q)]}_\E -\chi_\l (p, q))
\\&= proj{\gamma}{[s(p)_\l s(q)]}_\E \quad(\textit{as ${\gamma|}_\H \subset \H$ and ${proj|}_\H= 0$})\\&= {[proj\gamma s(p)_\l proj\gamma s(q)]}_\L = {[\bar{\gamma}(p)_\l \bar{\gamma}(q)]}_\L.\end{align*}
and
%\begin{align} P\bar{\gamma}(p) = P proj\gamma s (p) = proj R \gamma s (p)= proj\gamma(R s)(p) = proj\gamma(sP)(p)=\bar{\gamma}(P)(p) (\tept{as } \gamma|_\H \subset h \text{ and } proj|_\H = 0).\end{align}

\begin{align*}(\N \bar{\gamma} - \bar{\gamma} \N)(p) &= (\N proj\gamma s - proj\gamma s \N)(p) 
\\&= (proj\R \gamma s - proj\gamma s \N)(p)
\\&= proj\gamma(\R s - s\N)(p) = 0 \quad (\text{ as } {\gamma|}_\H \subset \H \text{ and } { proj |}_\H = 0).
\end{align*}
This shows that the map $\bar{\gamma}: \L_\N \to \L_\N$ is an automorphism of the Nijenhuis Lie conformal algebra $\L_\N$. In other words, $\bar{\gamma}\in Aut(\L_\N)$. Observe that, $\overline{\gamma_1 \gamma_2} = \overline{\gamma_1} \ \overline{\gamma_2}.$
Now we obtain a group homomorphism
$\Pi : Aut_\H(\E_\R ) \to Aut(\H_\Q) \times Aut(\L_\N)$ given by $\gamma \mapsto ({\gamma|}_\H, \bar{\gamma})$.
In general, we say any pair $(\a,\b)\in Aut(\H_\Q) \times Aut(\L_\N)$ is called inducible if $(\a,\b)$ lies in the image of $\Pi$, given above.
\begin{prop}\label{prop6.1}
Consider the non-abelian extension of Nijenhuis Lie conformal algebras. For a section $s$, suppose the extension corresponds to the non-abelian $2$-cocycle $(\chi_\l, \rho, \Phi)$. Then a pair $(\a, \b) \in \mathrm{Aut}(\H_ \Q) \times \mathrm{Aut}(\L_\N)$ of Nijenhuis Lie conformal algebra automorphisms is inducible if and only if there exists a $\mathbb{C}$-linear map $\eta : \L \to \H$ satisfying the following conditions:
\begin{align}
\a(\rho (p)_\l h) - \rho {(\b(p))}_\l \a(h) &= {[\eta(p)_\l \a(h)]}_\H, \label{eq:16} \\
\a(\chi_\l(p, q)) - \chi_\l(\b(p), \b(q)) &= \rho {(\b(p))}_\l \eta(q) - \rho {(\b(q))}_{-\p-\l} \eta(p) - \eta({[p_\l q]}_\L) + [\eta(p)_\l \eta(q)]_\H, \label{eq:17} \\
\a(\Phi(p)) - \Phi(\b(p)) &= \Q(\eta(p)) - \eta(\N(p)),   \label{eq:18}
\end{align} for all $p, q \in \L$ and $ h \in \H$.
\end{prop}
\begin{proof} Let $(\a, \b)$ be an inducible pair, i.e., there exists a Nijenhuis Lie conformal algebra automorphism $\gamma \in \mathrm{Aut}_\R(\E_\R)$ such that ${\gamma|}_\H = \a$ and $proj \gamma s = \b$. For any $p \in \L$, we observe that $(\gamma s - s \b)(p) \in \ker(proj)$, which in turn implies that $(\gamma s - s \b)(p) \in \H$. We define a map $\eta : \L \to \H$ by 
$\eta(p) := (\gamma s - s \b)(p), \quad \text{for } p \in \L.$
 
Then we observe that
\begin{align*}
 \a({\rho (p)}_\l h) - \rho {(\b(p))}_\l \a(h) =& \a({[{s(p)}_\l h]}_\E) - {[{s \b(p)}_\l \a(h)]}_\E
 \overbrace{=}^{\text{as }\a ={\gamma|}_\H}{[{\gamma s(p)}_\l \gamma(h)]}_\E - {[{s \b(p)}_\l \a(h)]}_\E 
\\=& {[{\gamma s(p)}_\l \a(h)]}_\E - {[{s \b(p)}_\l \a(h)]}_\E = {[{\eta(p)}_\l \a(h)]}_\H,
\end{align*}
for $p \in \L$ and $h \in \H$. Similarly, for any $p, q \in \L$, we get that

\begin{align*}
\rho {(\b(p))}_\l \eta(q) - \rho {(\b(q))}_{-\p-\l} \eta(p) &- \eta({[p_\l q]}_\L) + {[\eta(p)_\l \eta(q)]}_\H
\\&= {[{s \b(p)}_\l (\gamma s - s \b)(q)]}_\E - {[{s \b(q)}_{-\p-\l} (\gamma s - s \b)(p)]}_\E \\
&\quad - (\gamma s - s \b)({[p_\l q]}_\L) + {[(\gamma s - s \b)(p)_\l (\gamma s - s \b)(q)]}_\H \\
&= {[{s \b(p)}_\l \gamma s(q)]}_\E - {[{s \b(p)}_\l s \b(q)]}_\E - {[{s \b(q)}_{-\p-\l} \gamma s(p)]}_\E \\
&\quad + {[{s \b(q)}_{-\p-\l} s \b(p)]}_\E - \gamma s({[p_\l q]}_\L) + s{[\b(p)_\l \b(q)]}_\L \\
&\quad + {[{\gamma s(p)}_\l \gamma s(q)]}_\E - {[{\gamma s(p)}_\l s \b(q)]}_\E - {[{s \b(p)}_\l \gamma s(q)]}_\E + {[{s \b(p)}_\l s \b(q)]}_\E \\
&= \big({[{\gamma s(p)}_\l \gamma s(q)]}_\E - \gamma s({[p_\l q]}_\L)\big) + {[{s \b(q)}_{-\p-\l} s \b(p)]}_\E + s({[\b(p)_\l \b(q)]}_\L) \\
&= \gamma\big({[{s(p)}_\l s(q)]}_\E - s({[p_\l q]}_\L)\big) - \big({[{s \b(p)}_\l s \b(q)]}_\E - s({[{\b(p)}_\l \b(q)]}_\L)\big) \\
&= \a(\chi_\l(p, q)) - \chi_\l(\b(p), \b(q)),
\end{align*}
and also
\begin{align*}
 \a(\Phi(p)) - \Phi(\b(p))
 &= \a\big(\R(s(p)) - s(\N(p))\big) - \big(\R(s \b(p)) - s(\N \b(p))\big)\\&
\overbrace{=}^{\a = {\gamma|}_\H} \gamma\big(\R(s(p)) - s(\N(p))\big) - \R(s \b(p)) + s(\N \b(p))\\&  \overbrace{=}^{\R \gamma = \gamma \R ,\quad\N \b = \b \N} \R \gamma(s(p)) - \gamma s(\N(p)) - \R(s \b(p)) + s(\b \N(p))  \\
&= \R((\gamma s - s \b)(p)) - (\gamma s - s \b)(\N(p)) \\
&= \R(\eta(p)) - \eta(\N(p)) \\
&= \Q(\eta(p)) - \eta(\N(p)).
\end{align*}
%Hence, one direction of the proof follows.
Conversely, suppose a $\mathbb{C}$-linear map $\eta : \L \to \H$ exists which satisfies the identities \eqref{eq:16}, \eqref{eq:17}, and \eqref{eq:18}. First, observe that, using the section $s$, the $\mathbb{C}[\p]$-module $\E$ can be identified with $s(\L) \oplus \H$. Equivalently, any element $f \in \E$ can be uniquely written as $f = s(p) + h$, for some $p \in \L$ and $h \in \H$. Using this identification, we define a map $\gamma : \E \to \E$ by
\begin{align*}
 \gamma(f) = \gamma(s(p) + h) = s(\b(p)) + (\a(h) + \eta(p)), \quad \text{for } f = s(p) + h \in \E.
\end{align*}
Since $s$, $\a$, $\b$ are all injections, we can easily verify that $\gamma$ is also an injective map. The maps $\a$ and $\b$ are also surjective, which implies that $\gamma$ is also a surjective map. Thus, $\gamma$ becomes a bijective map. By using the conditions \eqref{eq:16} and \eqref{eq:17}, it is easy to show that $\gamma : \E \to \E$ is a Lie conformal algebra homomorphism. Moreover, for any $f = s(p) + h \in \E$, we observe that
\begin{align*}
(\gamma \circ \R)(f) &= (\gamma \circ \R)(s(p) + h) \nonumber \\
&= \gamma\big(\R s(p) + \Q(h)\big) \nonumber \\
&\overbrace{=}^{\Phi = \R s - s \N} \gamma\big(s \N(p) + \Q(h) + \Phi(p)\big)  \nonumber \\
&= s(\b \N(p)) + \a(\Q(h) + \Phi(p)) + \eta(\N(p)) \nonumber \\
&= s(\N \b(p)) + \a \Q(h) + \a \Phi(p) + \eta(\N(p)) \nonumber \\
&= \R(s \b(p)) - \Phi(\b(p)) + \a \Q(h) + \a \Phi(p) + \eta(\N(p)) \nonumber \\
&\overbrace{=}^{\a \Q = \Q \a}\R(s \b(p)) + \Q \a(h) + \big(\a \Phi(p) + \eta(\N(p)) - \Phi(\b(p))\big)  \nonumber \\
&= \R(s \b(p)) + \Q \a(h) + \Q(\eta(p)) \quad (\text{by } \eqref{eq:18} ) \nonumber \\
&= \R\big(s(\b(p)) + \a(h) + \eta(p)\big) \quad (\text{as } {\R|}_\H = \Q) \nonumber \\
&= (\R \circ \gamma)(f).
\end{align*}
Finally, it follows from the definition of $\gamma$ that $\gamma(\H) \subset \H$. Thus, $\gamma \in \mathrm{Aut}_\H(\E_\R)$. Additionally, it is easy to see that ${\gamma|}_\H = \a$ and $proj \gamma s = \b$. This shows that the pair $(\a, \b)$ is inducible.
    \end{proof}It is important to observe that the necessary and sufficient condition derived in the previous proposition depends on the choice of section $s$. Ideally, one would like to obtain a criterion that is independent of such choices. To achieve this, we adapt a method originally developed by Wells in the context of abstract groups \cite{Wells} , and apply it in the setting of Nijenhuis Lie conformal algebras.

Given any pair $(\a,\b)\in Aut(\H_\Q) \times Aut(\L_\N)$ of automorphisms of Nijenhuis Lie conformal algebra, we define a new triple $(\chi_\l^{(\a,\b)}, \rho^{(\a,\b)}, \Phi^{(\a,\b)})$ of $\mathbb{C}[\p]$-module maps
$$\chi_\l^{(\a,\b)}: \L\otimes \L \to \H[\l], \quad \rho^{(\a,\b)}: \L \otimes \H \to \H[\l], \quad \Phi^{(\a,\b)}: \L \to \H$$
by \begin{align}\label{newcocycle}
 \chi^{(\a,\b)}_\l(p, q) := \a \circ \chi_\l (\b^{-1}(p), \b^{-1}(q)), \quad \rho^{(\a,\b)}(p)_\l h := \a(\rho({\b^{-1}(p)})_{\l} \a^{-1}(h)), \quad \Phi^{(\a,\b)}(p) := \a \Phi(\b^{-1}(p)),
\end{align}
for $p, q \in \L$ and $h \in \H$. Then we have the following result.
\begin{lem}The triple $(\chi_\l^{(\a,\b)}, \rho^{(\a,\b)}, \Phi^{(\a,\b)})$ is a non-abelian $2$-cocycle on $\L_\N$ with values in $\H_\Q$.
\end{lem}
\begin{proof}
 The triple $(\chi_\l, \rho, \Phi)$ being a non-abelian $2$-cocycle implies that the identities \eqref{2cocycle1}, \eqref{2cocycle2}, \eqref{E1} and \eqref{E2} hold. In these identities, if we replace \(p\), \(q\), and \(h\) by \(\b^{-1}(p)\), \(\b^{-1}(q)\), and \(\a^{-1}(h)\), respectively, we simply get the non-abelian $2$-cocycle conditions for the triple \((\chi_\l^{(\a,\b)}, \rho^{(\a,\b)}, \Phi^{(\a,\b)})\). 
%For example, it follows from \eqref{E1} that:\begin{align} \a(\rho(P \b^{-1}(p))_\l Q \a^{-1}(h)) &= \a Q(\rho (P \b^{-1}(p))_\l \a^{-1}(h)) + \a Q [\Phi_{\b^{-1}(p)}, \a^{-1}(h)]_\H - \a[\Phi_{\b^{-1}(p)}, Q \a^{-1}(h)]_\H\\&= \a Q(\rho{(\b^{-1}(p))}_\l Q \a^{-1}(h)) - \a[\Phi {\b^{-1}(p)}, Q \a^{-1}(h)]_\H.\end{align}by using Eq. \eqref{newcocycle}\begin{align} (\rho^{(\a,\b)})(P(p))_\l Q(h) &= Q(\rho^{(\a,\b)}P(p)_\l h) + Q[\Phi^{(\a,\b)}(p), h]_\H - [\Phi^{(\a,\b)}(p), Q(h)]_\H\\&= \end{align}This shows that the identity \eqref{E1} also holds for the triple \((\chi_\l^{(\a,\b)}, \rho^{(\a,\b)}, \Phi^{(\a,\b)})\).
\end{proof}
\begin{defn}
A map $\mathfrak{W} : Aut(\H_\Q)\times Aut(\L_\N ) \to H^2_{nab}(\L_\N , \H_\Q)$ by
$\mathfrak{W}(\a,\b) = [(\chi_\l^{(\a,\b)}, \rho^{(\a,\b)}, \Phi^{(\a,\b)}) - (\chi_\l, \rho, \Phi)],$ the equivalence class of $(\chi_\l^{(\a,\b)}, \rho^{(\a,\b)}, \Phi^{(\a,\b)})- (\chi_\l, \rho, \Phi).$ The map $\mathfrak{W}$ is called Wells map.\end{defn} 

Note that non-abelian $2$-cocycles depend on the choice of section, but Wells map does not depend on the choice of section. This can be understood by the fact that equivalent $2$-cocycles belong to the same cohomology class as described in the previous section. Since Wells map is related to these cohomology classes rather than the specific cocycles, that is why it is independent of the choice of section. Thus, different sections do not affect the outcome of the Wells map, as they yield equivalent cocycles in cohomology. This can be stated and explained in the form of following proposition:
\begin{prop} Wells map is independent of the choice of section.
\end{prop}
\begin{proof} Let $ s'$ be any other section of the map $proj$, and let $(\chi_\l', \rho', \Phi')$ be the corresponding non-abelian $2$-cocycle. We have seen that the non-abelian $2$-cocycles $(\chi_\l, \rho, \Phi)$ and $(\chi_\l', \rho', \Phi')$ are equivalent by the map $\tau := s - s'$. Using this, it is easy to verify that the non-abelian $2$-cocycles $(\chi_\l^{(\a, \b)}, \rho^{(\a, \b)}, \Phi^{(\a, \b)})$ and $(\chi_\l'^{(\a, \b)}, \rho'^{(\a, \b)}, \Phi'^{(\a, \b)})$ are equivalent by the map $\a \tau \b^{-1}$. Combining these results, we observe that the $2$-cocycles $(\chi_\l^{(\a, \b)}, \rho^{(\a, \b)}, \Phi^{(\a, \b)}) - (\chi_\l, \rho, \Phi)$ and $(\chi_\l'^{(\a, \b)}, \rho'^{(\a, \b)}, \Phi'^{(\a, \b)}) - (\chi_\l', \rho', \Phi')$ are equivalent by the map $\a \tau \b^{-1} - \tau$. Therefore, their corresponding equivalence classes in $H^2_{\text{nab}}({\L}_\N, {\L}_\Q)$ are same. In other words, the map $\mathfrak{W}$ does not depend on the chosen section.
\end{proof}
\begin{thm}\label{thm6.4}
 Let $\L_\N$ and $\H_\Q$ be two Nijenhuis Lie conformal algebras and \begin{align*}
 0 \longrightarrow \H_\Q \xrightarrow{inc} \E_\R \xrightarrow{proj} \L_\N \longrightarrow 0
\end{align*} be a non-abelian extension of $\L_\N$ by $\H_\Q$ with the section $s$. The corresponding non-abelian $2$-cocycles of $\E_\R$ is denoted by $(\chi_\l,\rho,\Phi)$. A pair $(\a,\b) \in Aut(\H_\Q) \times Aut(\L_\N)$ is inducible if and
only if the non-abelian $2$-cocycles $(\chi_\l,\rho,\Phi)$ and $(\chi_\l^{(\a,\b)}, \rho^{(\a,\b)}, \Phi^{(\a,\b)})$ are equivalent. In other words, Wells map of a pair of automorphisms $(\a,\b)$ is zero.
\end{thm}
\begin{proof}
Let $(\a, \b) \in \mathrm{Aut}(\H_\Q) \times \mathrm{Aut}(\L_\N)$ be an inducible pair of Nijenhuis Lie conformal algebra automorphisms. For any fixed section $s$, let the given non-abelian extension produce the non-abelian $2$-cocycle $(\chi_\l, \rho, \Phi)$. Then by Proposition \ref{prop6.1}, there exists a $\mathbb{C}$-linear map $\eta : \L \to \H$ satisfying Eqs. \eqref{eq:16}-\eqref{eq:18}. In these identities, if we replace $p,q, h$ respectively by $\b^{-1}(p), \b^{-1}(q), \a^{-1}(h)$, we get that
\begin{align*}
 {\rho^{(\a,\b)} (p)}_\l h - {\rho (p)}_\l h &= [{\eta \b^{-1}(p)}_\l h]_\H, \\
\chi^{(\a,\b)}_\l (p, q) - \chi_\l (p, q) &= {\rho(p)}_\l \eta\b^{-1}(q) - {\rho(q)}_{-\p-\l} \eta \b^{-1}(p) - \eta \b^{-1}({[p_\l q]}_\L) + {[{\eta \b^{-1}(p)}_\l \eta \b^{-1}(q)]}_\H, \\
\Phi^{(\a,\b)}(p) - \Phi(p) &= \Q(\eta \b^{-1}(p)) - \eta \b^{-1}(\N(p)).
\end{align*}
This shows that the non-abelian $2$-cocycles $(\chi^{(\a,\b)}, \rho^{(\a,\b)}, \Phi^{(\a,\b)})$ and $(\chi_\l, \rho, \Phi)$ are equivalent by the map $\eta \b^{-1}$. This implies that the cohomology class 
$\mathfrak{W}(\a, \b) = [(\chi_{(\a,\b)}, \rho {(\a,\b)}, \Phi_{(\a,\b)}) - (\chi_\l \psi, \Phi)]$
is trivial.
\par Conversely, assume that $\mathfrak{W}(\a, \b) = 0$. As before, let $s$ be any section and $(\chi_\l, \rho, \Phi)$ be the non-abelian $2$-cocycle produced from the given non-abelian extension. Then it follows that the non-abelian $2$-cocycles $(\chi^{(\a,\b)}, \rho^{(\a,\b)}, \Phi^{(\a,\b)})$ and $(\chi_\l, \rho, \Phi)$ are equivalent (say by the map $\tau : \L \to \H$). Then it is easy to see that the map $\eta:= \tau \b: \L \to \H$ satisfies the Eqs. \eqref{eq:16}-\eqref{eq:18}. Therefore, by Proposition \ref{prop6.1}, the pair $(\a, \b)$ is inducible.
\end{proof}
In the Lie conformal algebra context, we can obtain Wells' exact sequence of Lie conformal algebras. Let's generalize the exact sequence of Wells maps in the context of Nijenhuis Lie conformal algebras. For any non-abelian extension of Nijenhuis Lie conformal algebras $0 \longrightarrow \H_\Q \xrightarrow{inc} \E_\R \xrightarrow{proj} \L_\N \longrightarrow 0$. Define a subgroup $Aut^{\L,\H}_{\H}(\E_\R) \subset Aut_{\H}(\E_\R)$ by
$Aut^{\L,\H}_{\H}(\E_\R):= \{\gamma\in Aut_{\H}(\E_\R) \mid \Pi(\gamma) = (Id_\H,Id_\L)\}.$ Thus, we have the following theorem.

\begin{thm}\label{thm6.5} Let $0 \longrightarrow \H_\Q \xrightarrow{i} \E_\R \xrightarrow{p} \L_\N \longrightarrow 0$ be a non-abelian extension of Nijenhuis Lie conformal algebras. Then there is an exact sequence
$1 \longrightarrow Aut^{\L,\H}_{\H}(\E_\R)\xrightarrow{inc} Aut_{\H}(\E_\R) \xrightarrow{\Pi} Aut(\H_\Q) \times Aut(\L_\N)\xrightarrow{\mathfrak{W}} H^2_{nab}(\L_\N,\H_\Q).$
\end{thm}
%\begin{proof} {\color{red} The proof of this theorem follows the same principle as Theorem 5.6 of \cite{DS}.}\end{proof}
%\subsection{ Declarations}
\subsubsection*{\textbf{Acknowledgement:}}The author would like to thank the editor and referees for their valuable input on our manuscript.
\subsubsection*{\textbf{Ethical Approval:}} Not applicable.
\subsubsection*{\textbf{Competing interests:}}
The authors have no competing interests to declare that are relevant to the content of this paper.
%\subsection*{Authors' contributions:} All authors contributed equally to the manuscript.
\subsubsection*{\textbf{Funding:}}This work is supported by the High-level Talent Research Start-up Project Funding of Henan Academy of Sciences (Project NO. 241819105).
\subsubsection*{\textbf{Availability of data and materials:}}
All data is available within the manuscript.
 
\end{document}